\documentclass[11pt]{article}

\usepackage{epsfig}
\usepackage{graphicx}
\usepackage{amsbsy}
\usepackage{amsmath}
\usepackage{amsfonts}
\usepackage{amssymb}
\usepackage{textcomp}
\usepackage{hyperref}
\usepackage{aliascnt}

\newcommand{\mcm}[3]{\newcommand{#1}[#2]{{\ensuremath{#3}}}} 

\mcm{\tuple}{1}{\langle #1 \rangle}
\mcm{\name}{1}{\ulcorner #1 \urcorner}
\mcm{\Nbb}{0}{\mathbb{N}}
\mcm{\Zbb}{0}{\mathbb{Z}}
\mcm{\Rbb}{0}{\mathbb{R}}
\mcm{\Cbb}{0}{\mathbb{C}}
\mcm{\Qbb}{0}{\mathbb{Q}}
\mcm{\Acal}{0}{\cal A}
\mcm{\Bcal}{0}{\cal B}
\mcm{\Ccal}{0}{\cal C}
\mcm{\Dcal}{0}{\cal D}
\mcm{\Ecal}{0}{\cal E}
\mcm{\Fcal}{0}{\cal F}
\mcm{\Gcal}{0}{\cal G}
\mcm{\Hcal}{0}{\cal H}
\mcm{\Ical}{0}{\cal I}
\mcm{\Jcal}{0}{\cal J}
\mcm{\Kcal}{0}{\cal K}
\mcm{\Lcal}{0}{\cal L}
\mcm{\Mcal}{0}{\cal M}
\mcm{\Ncal}{0}{\cal N}
\mcm{\Ocal}{0}{{\cal O}}
\mcm{\Pcal}{0}{{\cal P}}
\mcm{\Qcal}{0}{{\cal Q}}
\mcm{\Rcal}{0}{{\cal R}}
\mcm{\Scal}{0}{{\cal S}}
\mcm{\Tcal}{0}{{\cal T}}
\mcm{\Ucal}{0}{{\cal U}}
\mcm{\Vcal}{0}{{\cal V}}
\mcm{\Wcal}{0}{{\cal W}}
\mcm{\Xcal}{0}{{\cal X}}
\mcm{\Ycal}{0}{{\cal Y}}
\mcm{\Zcal}{0}{{\cal Z}}
\mcm{\Mfrak}{0}{\mathfrak M}

\mcm{\restric}{0}{\upharpoonright}
\mcm{\upset}{0}{\uparrow}
\mcm{\onto}{0}{\twoheadrightarrow}
\mcm{\smallNbb}{0}{{\small \mathbb{N}}}
\DeclareMathOperator{\preop}{op}
\mcm{\op}{0}{^{\preop}}

\newcommand{\se}{\subseteq}
{\begin{array}{c}
\setlength{\unitlength}{1em}}%
{\end{array}}

\usepackage{amsthm}

\newcommand{\theoremize}[2]{\newaliascnt{#1}{thm} \newtheorem{#1}[#1]{#2} \aliascntresetthe{#1}}

\theoremstyle{plain}
\newtheorem{thm}{Theorem}[section]
\theoremize{lem}{Lemma}
\theoremize{skolem}{Skolem}
\theoremize{fact}{Fact}
\theoremize{sublem}{Sublemma}
\theoremize{claim}{Claim}
\theoremize{obs}{Observation}
\theoremize{prop}{Proposition}
\theoremize{cor}{Corollary}
\theoremize{que}{Question}
\theoremize{oque}{Open Question}
\theoremize{con}{Conjecture}

\theoremstyle{definition}
\theoremize{dfn}{Definition}
\theoremize{rem}{Remark}
\theoremize{eg}{Example}
\theoremize{exercise}{Exercise}
\theoremstyle{plain}

\newcommand{\myref}[1]{\autoref{#1}~(\nameref{#1})}
\usepackage{verbatim}
\usepackage{enumerate}
\usepackage[all]{xy}

\usepackage{subfig}

\title{Local 2-separators
}

\author{Johannes Carmesin
\medskip 
\\
  {University of Birmingham}
}
\newcommand{\sm}{\setminus}

\DeclareMathOperator{\expl}{Ex_r}

\mcm{\Fbb}{0}{\mathbb{F}}

\begin{document}

\maketitle

\begin{abstract}

How can sparse graph theory be extended to large networks, where algorithms whose running time is 
estimated using the number of vertices are not good enough? 
I address this question by introducing `Local Separators' 
of graphs. 
Applications include:
\begin{enumerate}
 \item A unique decomposition theorem for graphs along their local 2-separators analogous to 
the 2-separator 
theorem;
\item an exact characterisation of graphs with no bounded subdivision of a wheel 
\cite{loc_ser_para}.
\end{enumerate}
\end{abstract}

\section{Introduction}

One of the big challenges in Graph Theory today is to develop methods and algorithms to study 
sparse 
large networks; that is, graphs where the number of edges is about linear in the number of 
vertices, 
and the number of vertices is so large that algorithms whose running time is estimated in terms of 
the vertex number are not good enough. Important contributions that provide partial results towards 
this big aim include the following. 

\begin{enumerate}
 \item 
{\bf Benjamini-Schramm limits of graphs.} Benjamini and Schramm introduced a notion of convergence 
of 
sequences of graphs that is based on neighbourhoods of vertices of bounded radius in 
\cite{BeSchrRec}. 
Applications of these methods include: testing for minor closed properties 
\cite{benjamini2010every} by Benjamini, Schramm  and Shapira or the proof of recurrence 
of planar graph limits by Gurel-Gurevich and Nachmias \cite{gurel2013recurrence}. 

\item
{\bf From Graphons to Graphexes.} Graphons have 
turned out to be a useful tool to study 
dense large networks \cite{{lovasz2012large},{lovaszarbeitsgemeinschaft}}. Motivated by 
these successes, analogues for sparse graph limits are proposed in 
\cite{{borgs2017sparse},{caron2017sparse},{herlau2016completely}}. 

\item
{\bf Graph Clustering.} 
The spectrum of the adjacency matrix of a graph can be used to 
identify large clusters, see the surveys \cite{von2007tutorial} or \cite{schaeffer2007graph}.

\item
{\bf Nowhere dense classes of graphs.} In their book \cite{nevsetvril2012sparsity},  
Ne{\v{s}}et{\v{r}}il and Ossana de 
Mendez 
systematically study a whole zoo of classes of sparse graphs and the relation between these classes.

\item
{\bf Refining tree-decompositions techniques.} 
Empirical results by Adcock, Sullivan and Mahoney suggest that some 
large networks do have tree-like structure \cite{adcock2013tree}. In \cite{adcock2016tree}, these 
authors say that: `Clearly, there is a need to develop Tree-Decompositions heuristics that are 
better-suited for the properties of
realistic informatics graphs'.
And they set the challenge to develop methods that combine the local and global structure of graphs 
using tree-decompositions methods. 
\end{enumerate}

Much of sparse graph theory -- in particular of graph minor theory -- is built upon the 
notion of a 
separator, which allows to cut graphs into smaller pieces, solve the relevant problems there and 
then stick together these partial solutions to global solutions. These methods include: 
tree-decompositions 
\cite{gmx}, 
the 
2-separator theorem and the block-cutvertex theorem, 
Seymour's decomposition theorem for regular matroids \cite{seymour1980decomposition}, as well as 
clique sums and rank width decompositions \cite{oum2005rank}. Understanding the relevant 
decomposition methods properly is fundamental 
to recent breakthroughs such as the Graph Minor Theorem \cite{MR2099147} or the Strong Perfect 
Graph Theorem \cite{robertson2006strong}.
As whether a given vertex set is separating depends on each vertex individually. So in the context 
of large networks it is unfeasible to test whether a set of vertices is separating. We believe that 
in order to extend such methods from sparse graphs to large networks, it is key to answer the 
following question. What are local separators of large networks?

Here, we answer this question. Indeed, we provide an example demonstrating that the naive 
definition of local separators misses key properties of separators. Then we introduce local 
separators of graphs that lack this defect. 
Our new methods have the following 
applications.

\begin{enumerate}[A)]

 \item A unique decomposition theorem for graphs along their local 2-separators analogous to 
the 2-separator 
theorem;
\item an exact characterisation of graphs with no bounded subdivision of a wheel. This 
 connects to  
direction (4) outlined above  \cite{loc_ser_para};
\item (work in progress) an analogue of the tangle-tree theorem of Robertson and Seymour, where the 
decomposition-tree 
is replaced by a general graph. This connects to direction (5).
\end{enumerate}

   \begin{figure} [htpb]   
\begin{center}
   	  \includegraphics[height=4cm]{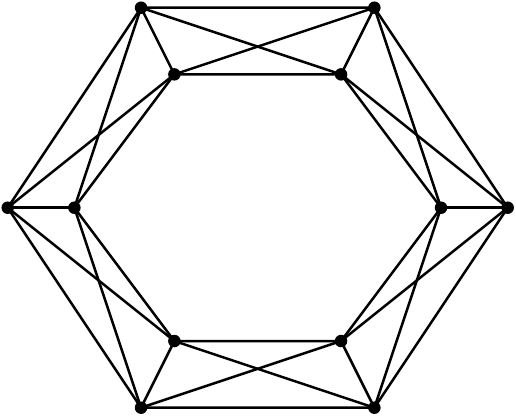}
   	  \caption{The graph $C_6 \boxtimes K_1$.}\label{fig:cycleK4}
\end{center}
   \end{figure}

\begin{eg}\label{eg:cycle-deco}
What is the structure of the graph in \autoref{fig:cycleK4}? According to the 
2-separator theorem, this graph is 3-connected and hence a basic graph that cannot be decomposed 
further. In this paper, however, we consider finer decompositions and according to our main 
theorem, 
the structure of this graph is: a family of complete graphs $K_4$ glued together in a cyclic way. 
 
\end{eg}

{\bf Our results.}
The 2-separator theorem\footnote{See {\cite[Section 4]{cunningham1980combinatorial}} for an 
overview of the history of 
the 2-separator theorem, see also \cite{seymour1980decomposition}. An 
alternative formulation of this theorem in terms of `2-sums' is given 
in \autoref{sec2}.} (in the strong form of Cunningham and Edmonds 
\cite{cunningham1980combinatorial}) says that every 2-connected graph has a unique minimal 
tree-decomposition of adhesion two all of whose torsos are 3-connected or cycles. 
We work with the natural extension of 
`tree-decompositions' where the decomposition-tree is replaced by an arbitrary graph. We refer to 
them as `graph-decompositions'.

Addressing the challenge set by 
Adcock, Sullivan and Mahoney, our main result is the following local strengthening of 
the 2-separator theorem.

\begin{thm}\label{thm:main_intro}
For every $r\in \Nbb\cup\{\infty\}$, every connected $r$-locally 2-connected graph $G$ has a 
graph-decomposition of adhesion two and locality $r$ such that all its torsos are $r$-locally 
3-connected or cycles of length at most $r$. 

Moreover, the separators of this graph-decomposition are the $r$-local 2-separators of $G$ that do 
not cross any other $r$-local 2-separator. 
\end{thm}

A key step in the proof of \autoref{thm:main_intro} is the following result, which seems to be of 
independent interest.
This can be seen as a local analogue of the fact 
that any 2-connected graph that is not 3-connected in which any 2-separator is crossed is a cycle.

\begin{thm}\label{structure_torso_intro}\label{structure_torso}
Let $r\in \Nbb\cup\{\infty\}$ and let $G$ be a connected graph that is $r$-locally 2-connected.
Assume that every $r$-local 2-separator of $G$ is crossed by an $r$-local 2-separator. 
Then $G$ is $r$-locally $3$-connected or a cycle of length at most $r$.
\end{thm}

\vspace{.3cm}

Beyond applications (A) to (C) mentioned above, this research includes the following 
applications.

\begin{enumerate}[A)]\setcounter{enumi}{3}
\item 
An algorithmic advantage of our main theorem is that the parallel running time of the corresponding 
algorithm does not 
depend on the number of vertices of the graph but just on the local structure\footnote{Indeed, 
we prove that for any pair of crossing local 2-separators there must be a cycle of length at most 
$r$ through their vertices.}; 
and we expect that 
our novel tool will be useful to study large networks. Indeed, a consequence of 
\autoref{thm:main_intro} is that one can pick the local 2-separators greedily, and all maximal 
graph-decompositions constructed in that way are essentially the same; in the sense that they 
contain the minimal graph-decomposition and additionally only have a few insignificant local 
2-separators within cycles of length at most $r$.

 \item Covers are important tools in Topology \cite{Hatcher} and Group Theory 
\cite{{serre_trees},{bass1993covering}}.
For covers of graphs, we refer the reader to the book 
\cite{gross2001topological} or the recent survey \cite{kwak2007graphs}.
Recent work includes \cite{benjamini2016structure},
\cite{benjamini2018structure}, \cite{fiala20183} and \cite{georgakopoulos2017covers}. 
The universal cover of a connected graph $G$ is always a tree and covers $G$. 
The 
\emph{$r$-local cover}, which is obtained by relaxing all cycles not generated by cycles 
of length at most $r$,  is covered by the universal cover but covers $G$. 
Our $r$-local 2-separator theorem lifts to the $r$-local cover of $G$, characterising 
the torsos of the 2-separator theorem of the cover as being the torsos of the 
$r$-local 2-separator theorem of $G$.  

\item Local tree-decompositions are considered in  \cite{grohe2003local} and 
\cite{frick2001deciding}. Here (and in the follow-up work \cite{locksepr} for arbitrary local 
separators), we offer tools to unify such collections of local tree-decompositions to a single 
graph-decomposition displaying the global structure of the graph. 
 
 \item Tree-decompositions have been used to study Cayley graphs of groups and other highly 
symmetric objects \cite{{hamann2017classification},{hamann2013classification}}. However, these tools 
were most helpful for infinite groups  as 
finite groups do not look like trees (roughly speaking). The graph-decompositions we construct here 
are invariant under the group of automorphisms and we expect that they can be used as a 
combinatorial tool to study  geometric properties of finite groups.  
\end{enumerate}

The remainder of this paper is structured as follows. 
In \autoref{sec2} we give an alternative formulation of \autoref{thm:main_intro}, and start 
explaining basic concepts, which we continue in \autoref{sec4}.

We 
invite all readers to look at \autoref{sec:blockcut} just after \autoref{sec4}. 
Indeed, in there we prove a local strengthening of the block-cutvertex theorem. This is a 
straightforward exercise, and it is not used in the rest of the paper. However, we believe it helps 
to digest 
the rest of the paper.  

In \autoref{sec:ex}, we prove an interesting special case of our main result (the parts of the 
proof that are not needed in our proof of \autoref{thm:main_intro} are put into \autoref{ex2}). 
Before proving \autoref{structure_torso_intro} in \autoref{sec:crossed}, we do some preparation in 
\autoref{props}. 

In \autoref{sec:uni}, we prove \autoref{minimal_N}, which implies \autoref{main_2sepr-intro}, a 
variant of \autoref{thm:main_intro}. 
Graph-decompositions are introduced in \autoref{sec:graph-deco}, and we conclude this section 
by deducing \autoref{thm:main_intro} from \autoref{minimal_N}.
Finally, in \autoref{sec:out} we discuss 
directions for further research. To make it easier for the reader to navigate through this paper, 
we added important definitions of this paper to the \lq table of content\rq, allowing readers to 
hyperlink to them in the pdf via the table of contents. 
Throughout the paper we fix a parameter $r\in \Nbb\cup\{\infty\}$.

\section{Constructive perspective}\label{sec2}

In this section we give an alternative formulation for \autoref{thm:main_intro} and define some 
basic notions for this paper. 

The 2-separator theorem can be stated in the decomposition version (as we did in the Introduction) 
as well as as the `constructive version'. For technical reasons we find it easier to work with the 
constructive version in the proofs and we will deduce the decomposition version in 
\autoref{sec:graph-deco}. We start by explaining the constructive version in this section. 

We recall the classical 2-separator theorem in the constructive version in full detail. This 
theorem has two aspects, the existential statement (which is the 
easy bit), and the 
uniqueness statement. 
The \emph{existential statement} says that every 2-connected graph $G$ can be constructed from 
3-connected graphs and cycles via 2-sums [\cite{oxley2},\S 83]. Clearly, 2-sums commute. Hence this 
sum is uniquely 
determined by the set of those summands that are basic; that is, they do not arise as a 2-sum 
of other summands. We refer to the set of basic summands as a \emph{decomposition} for 
$G$. 
We say that one decomposition for $G$ is 
\emph{coarser} (or smarter) than another if it has the same set of 3-connected graphs and its 
cycles can be build from cycles of the other decomposition (via the implicitly defined 
2-sums).
The \emph{uniqueness statement} says that there is a decomposition for $G$ with the universal 
property that it is coarser than any other decomposition for $G$.

\vspace{.3cm}

In analogy to 2-sums, we introduce the 
notion of \emph{$r$-local 2-sum}. This notion includes the usual 2-sums operation but additionally 
one is allowed to glue along edges of the same graph -- as long as they have distance at 
least $r$ (roughly speaking). We also introduce local 1-separators and local 2-separators and 
essentially\footnote{See \autoref{sec4} below for the complete definition.} define that a graph is 
\emph{locally 2-connected} if it has no local 1-separator; and 
`locally 
3-connected' is defined analogously.  All these terms carry the parameter `$r$' that measures how 
local this is (when the precise value of the parameter is not clear from the context, we shall 
write `$r$-local' in place of just `local'). The constructive version of \autoref{thm:main_intro}  
is the 
following. 

\begin{thm}\label{main_2sepr-intro}
 Every $r$-locally 2-connected graph can be constructed via $r$-local 2-sums from 
$r$-locally 3-connected graphs and cycles of length $\leq r$. 

There is such an $r$-local decomposition with the universal property that it is coarser than any 
other $r$-local decomposition for $G$. 
\end{thm}

\begin{rem}
 As for the classical 2-separator theorem, our local 2-separator theorem has two parts; the first 
sentence gives the existential statement and the second is the uniqueness statement. The uniqueness 
statement is more difficult to prove. 
\end{rem}

\vspace{.3cm}

We continue by defining some of the basic notions for this paper rigorously. 
How do we define local cutvertices? 
Roughly, a vertex should be a local cutvertex if the ball around it gets disconnected after its 
removal. But which definition of ball do we take? Do we take the definition where we allow edges in 
the ball joining two vertices of maximum distance or not? Answer: we take both definitions, as 
formalised in \autoref{dfn:ball}. Informally, the ball around a vertex consists of all edges and 
vertices on closed walks of bounded length starting at that vertex.

\begin{dfn}\label{dfn:ball}
 Given a graph $G$ with a vertex $v$ and an integer $s$, the \emph{ball} of radius $s$ around the 
vertex $v$ is the induced subgraph of $G$, whose vertices are those of distance at most $s$ from 
$v$ and without all edges joining two vertices of distance precisely $s$.
Similarly, given a half-integer $s+\frac{1}{2}$, the \emph{ball} of radius $s+\frac{1}{2}$ around 
the 
vertex $v$ is the induced subgraph of $G$, whose vertices are those of distance at most $s$ from 
$v$. We denote the ball of radius $s$ around $v$ by $B_s(v)$. 
Below we will often consider the graph $B_s(v)-v$, to which we refer as a \emph{punctured ball}.
Given a parameter $r\in \Nbb\cup\{\infty\}$, a vertex $v$ is an \emph{$r$-local cutvertex} of $G$ 
if it 
separates the ball of radius $r/2$ around $v$; formally: $B_{r/2}(v)-v$ is 
disconnected.\addcontentsline{toc}{subsection}{Dfn: ball}
\end{dfn}

\begin{lem}\label{gen}
Given a parameter $r\in \Nbb$ and a graph, all cycles of the subgraph $B_{r/2}(v)$ are generated 
by the cycles of 
length at most $r$.  
\end{lem}

\begin{proof}
 Construct a spanning tree of $B_{r/2}(v)$ rooted at $v$ so that a vertex 
with distance $d\leq r/2$ in $G$ from $v$ has distance $d$ in the spanning tree; this 
can easily be done by induction building the spanning tree layer by layer. Every fundamental cycle 
of this spanning tree has length at most $r$ (if $r$ is even, note that there is no edge between 
two vertices of distance $r/2$ from $v$. And if $r$ is odd note that there are no vertices of 
distance exactly $r/2$ from $v$). As every cycle is generated by the fundamental cycles, the cycles 
of length at most $r$ generate. 
\end{proof}

\begin{rem}
The bound $r$ in \autoref{gen} is sharp as can be seen by considering graphs $G$ 
that are equal to cycles of length $r$. 
\end{rem}

Informally speaking, the `2-sums operation' on graphs can be seen as the inverse operation of 
cutting along 2-separators and taking torsos. In the following we will introduce a  
local version of the `2-sums operation' on graphs.

\begin{dfn}[Local 2-sum]
 Given a family of weighted graphs $(G_i|i\in [n])$ and a set of weighted directed edges $e_i$ of 
$G_i$,
the \emph{local 2-sum} of this family is the graph obtained from the disjoint union of the 
set of 
graphs $\{G_i|i\in [n]\}$ by identifying the start-vertices of the edges $e_i$, and the 
terminal vertices 
of the edges $e_i$, and then deleting all edges $e_i$. 
For this local 2-sum to be valid, it must further satisfy the following condition for each $i\in 
[n]$. For each $i\in [n]$, we denote by $\gamma_i$ the length of the shortest path between the 
endvertices of the edge $e_i$ in the graph $G_i-e_i$. By $\delta_i$ we denote the minimum of the 
values $\gamma_j$ for $j\neq i$. 
We now further require that the length of the edge $e_i$ is equal to 
$\delta_i$.\addcontentsline{toc}{subsection}{Dfn: Local 2-sum}
\end{dfn}

\begin{rem}
 We stress that the graphs $G_i$ just form a family, so some of them may coincide, but the edges 
$e_i$ 
form a set, so they must all be distinct. In the disjoint union of the set of graphs $G_i$ we only 
have one copy for every graph, no matter how often it appears in the family. 
\end{rem}

Often, we will not explicitly specify a direction of the edges $e_i$ but assume it is given 
implicitly by the context or just take an arbitrary choice.

We say that a local $2$-sum is \emph{$r$-local} if any pair consisting of two starting-vertices or 
two terminal vertices, respectively, of edges $e_i$ and $e_j$ that live 
in the same host graph $G_i=G_j$ have distance at least $r+1$.

\begin{rem}
While in this section we have been working with graphs whose edges are assigned positive integer 
lengths 
bounded by $r$, in the rest of the paper all graphs have no weights on their edges. This is 
essentially the same; indeed, to get from 
such a weighted graph to a genuine graph just replace each weighted edge by a path of the same 
length. 
\end{rem}

\section{Explorer neighbourhood}\label{sec4}

In this section we define local 2-separators and explain the motivation behind our definition. 

The notion of local 1-separators has been explained above. But how should one define local 
2-separators? The first thing is that perhaps one only might want to consider pairs of vertices as 
local 2-separators if they have bounded distance between them. Indeed, otherwise if they were 
separating we would rather like to think about them as each being a local 1-separator. Okay, so we 
have a pair $(v,w)$ of vertices of bounded distance that separates their neighbourhood. But how do 
we define their neighbourhood precisely? Something that looks almost right is just picking one of 
the vertices arbitrarily and taking a ball around them. More precisely, one could require that 
$B_{r/2}(v)-v-w$ is disconnected for some parameter $r$. However, it could be that when we swap $v$ 
and 
$w$ then it switches from disconnected to connected. So perhaps the next attempt 
would be to take $(B_{r/2}(v)\cup B_{r/2}(w))-v-w$; just to make it symmetric in $v$ and $w$. Below 
we will 
refer to this long expression as the \emph{punctured double-ball}. 

The disadvantages of this definition, although almost correct, are more subtle. The main reason is 
perhaps that with that definition our proofs do not seem to work, as important structural 
properties are simply not true. Indeed, with this 
definition \myref{corner-lemma} does not work. 
This lemma is a natural generalisation of a lemma for usual separators, and we believe that any 
natural notion of local separators should have this property. 
The 
reason why that lemma is not true in this case is that 
the double-ball $B_{r/2}(v)\cup 
B_{r/2}(w)$ may contain cycles that are composed of a path from the ball $B_{r/2}(v)$ and from 
$B_{r/2}(w)$ but 
are not a cycle of either of these two balls, see \autoref{double-ball}.

   \begin{figure} [htpb]   
\begin{center}
   	  \includegraphics[height=4cm]{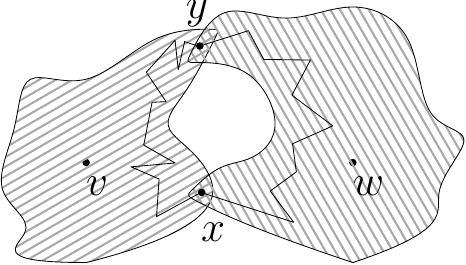}
   	  \caption{The balls $B_{r/2}(v)$ and $B_{r/2}(w)$ are marked by grey stripes, 
in rising and falling patters, respectively. Two paths between the vertices $x$ 
and $y$, one from either ball, form a cycle that is contained in neither ball.} \label{double-ball}
\end{center}
   \end{figure}

Informally speaking, the definition we take is similar to the double ball $B_{r/2}(v)\cup 
B_{r/2}(w)$ and actually agrees with it up to distance $r-d$, where $d$ is the distance between the 
vertices $v$ and $w$ -- but towards the boundary it `gets more fuzzy'. We will call our notion of 
neighbourhood `explorer-neighbourhood' and think about it as follows:
imagine two explorers discovering the graph starting from the vertices $v$ and $w$ respectively,  
with the goal of 
separately discovering the graph and at the end combining their maps of the balls $B_{r/2}(v)$ and 
$B_{r/2}(w)$ into a single map. First they discover all shortest paths between the 
vertices 
$v$ and $w$ together and put them on the common map. 
 We refer to the set of vertices on 
these paths as the \emph{core}. Then they return to 
their respective starting vertices and start 
exploring the graph from there up to 
distance $r/2$. On their map they mark each vertex  by the set of shortest paths to that vertex 
from  the core (within their respective balls). There may be vertices with distance $r/2$ from 
the core that have distance at 
most $r/2$ to 
the vertex $v$ but a larger distance to the vertex $w$. Such vertices are only discovered by the 
explorer based at $v$. There may also be vertices $u$ discovered by both explorers. However they 
might not discover a common shortest path to that vertex. In this case there will be two copies 
of that vertex in the explorer-neighbourhood, while 
there is only one copy in the double ball $B_{r/2}(v)\cup B_{r/2}(w)$.

\begin{dfn}[Explorer-neighbourhood] 
 Now we give a formal definition of the \emph{explorer-neighbourhood} of parameter $r$ in a 
graph $G$ with 
explorers 
based at the vertices $v$ and $w$ with distance\footnote{In this paper the 
explorer-neighbourhood of vertices $v$ and $w$ of distance more than $\frac{r}{2}$ is undefined; 
and hence throughout the paper in statements where the explorer-neighbourhood is mentioned we have 
implicitly the assumption that the involved vertices have distance at most $\frac{r}{2}$.} at most 
$\frac{r}{2}$. The \emph{core} is the set of all vertices on shortest 
paths between the 
vertices $v$ and $w$.  
We take a copy of the ball $B_{r/2}(v)$ where we label a vertex $u$ with the set 
of shortest paths from the core to $u$ contained in the ball $B_{r/2}(v)$. Similarly, we 
take a copy of the ball $B_{r/2}(w)$ where we label a vertex $u$ with the set 
of shortest paths from the core to $u$ contained in the ball $B_{r/2}(w)$. Now we take the union 
of 
these two labelled balls -- with the convention that two vertices are identified if they have a 
common label in their sets; that is, there is a shortest path from the core to that vertex 
discovered by both explorers. 
(Note that the same vertex $x$ of $G$ could be in both balls but the label sets could be 
disjoint, see \autoref{fig:exex}. In this case there would be two copies of that vertex in the 
union. 
In such a case the 
union would not be a subgraph of the original graph).
We denote the explorer neighbourhood by $\expl(v,w)$. This completes the definition of 
explorer neighbourhood. \addcontentsline{toc}{subsection}{Dfn: Explorer-neighbourhood}
\end{dfn}

   \begin{figure} [htpb]   
\begin{center}
   	  \includegraphics[height=3cm]{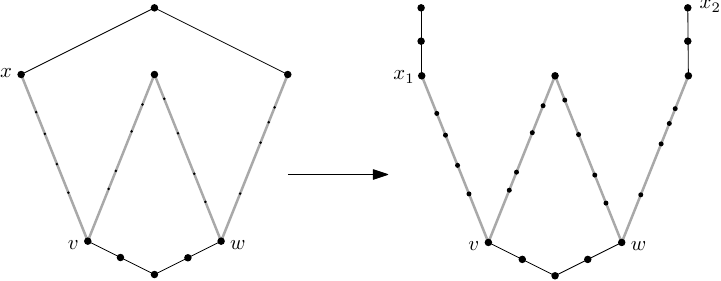}
   	  \caption{On the right we depicted the explorer-neighbourhood $\expl(v,w)$ of the graph on 
the left. The value for $r/2$ is seven. 
Here the grey paths all have length equal to $(r/2)-2$.
The core is just the path of length four between $v$ and $w$. The 
cycle of 
length $r$ is still a cycle in $\expl(v,w)$ since as a cycle it is included in both $B_{r/2}(v)$ 
and 
$B_{r/2}(w)$, see \autoref{unique_copy_extended} for details. The cycle of length $r+2$ is not 
contained in one of the balls $B_{r/2}(v)$ or $B_{r/2}(w)$ and hence some of its vertices get two 
copies in 
$\expl(v,w)$. 
Indeed, the vertex $x$ has distance at most $r$ from both vertices $v$ and $w$. Still it has the 
two copies 
$x_1$ and $x_2$ in the explorer-neighbourhood. 
} \label{fig:exex}
\end{center}
   \end{figure}

   \begin{lem}\label{unique_copy}
Given two vertices $a_1$ and $a_2$ of distance at most $r/2$, all vertices on shortest paths 
between 
$a_1$ and $a_2$ and edges incident with such vertices have unique copies in 
$\expl(a_1,a_2)$. 

In particular, edges incident with vertices of the core have unique copies in $\expl(a_1,a_2)$.
\end{lem}

\begin{proof}
 By definition vertices on shortest paths between $a_1$ and $a_2$ have unique 
copies in 
$\expl(a_1,a_2)$. 
Let $e$ be an edge one of whose endvertices is in the core. 
If both endvertices of $e$ are in the core, $e$ has a unique copy.
Otherwise, the edge $e$ is a shortest path from the core to the other endvertex. 
Clearly, the edge $e$ is in one of the balls $B_{r/2}(a_1)$ or $B_{r/2}(a_2)$.
If it is in both balls, then the two copies of its endvertex must agree as they are both labelled 
with the edge $e$. 
 
 The `In particular'-part follows immediately. 
\end{proof}

\begin{eg}
\autoref{unique_copy} implies that neighbours of vertices in the core have unique copies \lq most 
of the time\rq. 
Here we give an example of a graph where neighbours of vertices in the core do not have unique 
copies. 
Let $C$ be a cycle of length $r+1$, where $r$ is an even number. Let $a_1$ and $a_2$ be two 
vertices on $C$ of distance $\frac{r}{2}$. Then the neighbours of $a_{i}$ with distance 
$\frac{r}{2}$ to $a_{i+1}$ do not have unique copies in $\expl(a_1,a_2)$; indeed, $\expl(a_1,a_2)$ 
is a path of length $\frac{3r}{2}$.
\end{eg}

\begin{lem}\label{unique_copy_extended}
Let $o$ be a cycle (or more generally a closed walk) of length at most $r$ containing 
vertices $a_1$ and $a_2$. Vertices of $o$ have unique copies in $\expl(a_1,a_2)$.
\end{lem}

\begin{proof}
Let $o$ be a closed walk as in the statement of the lemma, and let $x$ be an arbitrary vertex on 
$o$. 
Let $S$ be a shortest path from $x$ to the core in the underlying graph (not just some subballs).
We will show that $S$ is completely included in both balls $B_{r/2}(a_1)$ and $B_{r/2}(a_2)$.
By symmetry, it suffices to show that $S$ is completely included in $B_{r/2}(a_1)$.

For any pair of vertices of the set $\{a_1,a_2,x\}$, pick a shortest path between these vertices. 
Let $o'$ be the closed walk obtained by concatenating these three paths. Let $y$ be the endvertex 
of the path $S$ on the core. 
We can pick, and we do pick, the shortest path between $a_1$ and $a_2$ so that it contains the 
vertex $y$. Hence the vertex $y$ is on the closed walk $o'$. 
As the closed walk $o$ also contains the vertices $a_1$, $a_2$ and $x$, its length is at least that 
of the closed walk $o'$; that is, the closed walk $o'$ has length at most $r$. 

Let $o''$ be the closed walk obtained by concatenating a shortest path from $a_1$ to $x$, the path 
$S$ and a shortest path from the vertex $y$ to $a_1$. Such a closed walk $o''$ can be obtained from 
the closed walk $o'$ by replacing a subwalk from $x$ via $a_2$ to $y$ by the path $S$. As $S$ is 
a shortest path between its endvertices, the length of $o''$ is at most that of $o'$; and thus at 
most $r$. Hence the closed walk $o''$ is completely contained within the ball $B_{r/2}(a_1)$ around 
$a_1$. Thus the shortest path $S$ is contained in that ball. As $S$ was chosen arbitrarily, every 
shortest path from $x$ to the core is included in the ball $B_{r/2}(a_1)$. By symmetry, the same is 
true for `$a_2$' in place of `$a_1$'.
Thus $x$ has a unique 
copy in the explorer-neighbourhood $\expl(a_1,a_2)$.

\end{proof}

The balls $B_{r/2}(v)$ and $B_{r/2}(w)$ are embedded within the explorer 
neighbourhood by construction. We refer to these embedded balls as $\iota(B_{r/2}(v))'$ 
and 
$\iota(B_{r/2}(w))'$, or simply $B_{r/2}(v)'$ and 
$B_{r/2}(w)'$ if the embedding map $\iota$ is clear from the context.

\begin{rem}\label{new_rem}
Every cycle of $\expl(a_1,a_2)$ of length at most $r$ containing one of the vertices 
$a_1$ or $a_2$, say $a_1$, is a cycle of $G$. Indeed, it is contained in the ball of radius $r/2$ 
around $a_1$ and as such a cycle of $G$. This can be seen as a converse of 
\autoref{unique_copy_extended}, and we shall use this observation in various places throughout the 
paper. 
\end{rem}

\begin{lem}\label{cycle_gen}
 Every cycle $o$ of the explorer neighbourhood $\expl(v,w)$ is generated from the cycles of the 
embedded balls $B_{r/2}(v)'$ and $B_{r/2}(w)'$.
\end{lem}

\begin{proof}
Each vertex of the cycle $o$ is a vertex of $B_{r/2}(v)'$ or $B_{r/2}(w)'$. We mark it with the 
respective vertex $v$ or $w$; and if it is in both, we mark it with both vertices $v$ and $w$. 
For each vertex $x$ on the cycle $o$ marked by a vertex $y\in \{v,w\}$, we pick a shortest path 
from $x$ to the core within the ball $B_{r/2}(y)'$. If a vertex is marked with $v$ and $w$ by the 
definition of explorer-neighbourhood, then we can assume, and we do assume, that we picked the same 
path 
for $y=v$ and $y=w$. 

Now for each edge $e\in o$ we construct a closed walk $o_e$ as follows. Start with $e$ and the two 
paths 
chosen at either endvertex of $e$, then join their endvertices in the core by a path within the 
core (which is connected by construction). 
Since for each edge $e$ of $o$, there is a mark $y\in \{v,w\}$ that is present at 
both endvertices of edge $e$, the closed walk $o_e$ is contained in $B_{r/2}(v)'$ or 
$B_{r/2}(w)'$.

Our aim is to generate the cycle $o$ from cycles of $B_{r/2}(v)'$ and $B_{r/2}(w)'$. For that we 
first add 
to $o$ the sum of all the cycles $o_e$ ranging over 
all $e\in o$ (taken over the binary field $\Fbb_2$). This sum takes only non-zero entries at edges 
of the core. As the core is a subset of $B_{r/2}(v)'\cap B_{r/2}(w)'$, the remainder is generated 
from the common 
cycles of  $B_{r/2}(v)'$ and $B_{r/2}(w)'$.
\end{proof}

\begin{dfn}[Local separators]
Given a graph $G$ with distinct vertices $v$ and $w$, we say that the set $\{v,w\}$ is an 
\emph{$r$-local 2-separator} if the punctured explorer-neighbourhood 
$\expl(v,w)-v-w$ is disconnected, and the vertices $v$ and $w$ have distance at most $r/2$ in the 
graph $G$. 

A connected graph is \emph{$r$-locally 2-connected} if it does not have an $r$-local cutvertex and 
it has a cycle of length at most $r$.
So there are no $r$-locally $2$-connected graphs for $r<3$. 
A graph is \emph{$r$-locally $2$-connected} if all its components are $r$-locally $2$-connected.

A connected $r$-locally 2-connected graph is \emph{$r$-locally 3-connected} if it does not have an 
$r$-local $2$-separator 
and it has 
at least four vertices.
A graph is \emph{$r$-locally $3$-connected} if all its components 
are $r$-locally $3$-connected. \addcontentsline{toc}{subsection}{Dfn: Local separators}
\end{dfn}

\begin{eg}
 A cycle of length $r+1$ is not $r$-locally $2$-connected. Moreover if it has more than three 
edges, any of its vertices together with any of its neighbours forms an $r$-local 
$2$-separator. 

Cycles of lengths at most $r$ are $r$-locally $2$-connected and not $r$-locally 3-connected if they 
have at least four edges. 
\end{eg}

\begin{lem}
An $r$-locally $3$-connected graph $G$ is $r$-locally $2$-connected.
\end{lem}

\begin{proof}
Assume that the graph $G$ is connected, has at least four vertices and contains a cycle. 
We are to show that if $G$ has an $r$-local cutvertex, then $G$ has an $r$-local $2$-separator.
So let $x$ be an $r$-local cutvertex. If $x$ is not contained in any cycle, $x$ is a genuine 
cutvertex of the graph $G$. 
As the graph $G$ is not a star by assumption, the vertex $x$ has a neighbour so that $x$ together 
with that neighbour is a $2$-separator of $G$. So $G$ has an $r$-local $2$-separator. 

So we may assume that $x$ is contained in a cycle $o$. Let $y$ be a neighbour of the vertex $x$ on 
the cycle $o$. 
We claim that $\{x,y\}$ is an $r$-local $2$-separator. 
Let $C$ be a component of the punctured ball $B_{\frac{r}{2}}(x)-x$ that does not contain the vertex 
$y$. 
Let $W$ be the set of edges incident with the vertex $x$ and the other endvertex in the component 
$C$. 
Let $z$ be a neighbour of $x$ in $C$. 

Suppose for a contradiction that the punctured explorer-neighbourhood $\expl(x,y)-x-y$ is 
connected. 
Then in particular, $\expl(x,y)-x$ is connected. So there is a path $P$ in there from $y$ to $z$. 
Then $P+xz$ is a cycle traversing the edge set $W$ precisely once. By \autoref{gen} and 
\autoref{cycle_gen} $P+xz$ is generated by cycles of length at most $r$ over $\Fbb_2$; hence one of 
these cycles intersects the edge set $W$ oddly. In particular, this cycle contains the vertex $x$, 
so it is a cycle of the ball $B_{\frac{r}{2}}(x)$. So we found a cycle of $B_{\frac{r}{2}}(x)$ that 
intersects the cut $W$ oddly. This is a contradiction. Hence $\expl(x,y)-x-y$ is disconnected; and 
so $\{x,y\}$ is an $r$-local $2$-separator.
\end{proof}

In a sense the next lemma says that local 2-components sitting at a local 2-separator are local 
(in that they contain a short path between the neighbours of the two separating vertices).

\begin{lem}[Local 2-Connectivity Lemma]\label{local_is_very_local}
Let $\{v,w\}$ be an $r$-local 
2-separator in an $r$-locally 2-connected graph $G$.
For every connected component $k$ of the punctured explorer-neighbourhood $\expl(v,w)-v-w$, 
there is a cycle $o'$ of length at most $r$ containing the vertices $v$ and $w$, and 
$o'$ contains a vertex of the component $k$ and 
$o'$ contains an edge incident with $v$ whose other endvertex is a vertex not in $k$. 
\end{lem}

\begin{proof}
 Let $k=k_1$ be an arbitrary component of the punctured explorer-neighbourhood $\expl(v,w)-v-w$, 
and 
let $k_2$ be the union of all 
other components of the punctured explorer-neighbourhood $\expl(v,w)-v-w$, which is nonempty as 
$\{v,w\}$ is a local 
2-separator. 
If one component of $\expl(v,w)-v-w$ had only one of the vertices $v$ and $w$ in its neighbourhood, 
then that vertex would be a local cutvertex. However, this is not possible as $G$ is $r$-locally 
2-connected by assumption. Hence all components of $\expl(v,w)-v-w$ have both vertices $v$ and $w$ 
in their neighbourhood. 
In particular, the vertex $v$ is adjacent to 
vertices of $k_1$ and $k_2$. 

Let $x_i$ be an arbitrary neighbour of the vertex $v$ in $k_i$ (for $i=1,2$). As the graph $G$ is 
$r$-locally 
2-connected, the vertex $v$ is not a cutvertex of the ball 
$B_{r/2}(v)$. So there is a path $P$ included in that ball from $x_1$ to $x_2$ avoiding $v$.
Let $o$ be the cycle obtained from $P$ by adding the vertex $v$. 
By \autoref{gen}, the cycle $o$ is generated from cycles of the ball $B_{r/2}(v)$ of length at most 
$r$. Consider the set $\Ccal$ of these cycles that contain the vertex $v$. As $o$ has precisely 
one edge to $k_1$ incident with $v$, 
there must 
be a cycle $o'$ in $\Ccal$ that contains an odd number of edges to $k_1$ 
incident with $v$. As the cycle $o'$ has maximum degree two, it contains precisely one edge to  
$k_1$ 
incident with $v$. The other edge of $o'$ incident with $v$ has its other endvertex in $k_2+w$. 
This completes the proof.
\end{proof}

\begin{rem}
 The bound $r$ for the cycle $o'$ in \myref{local_is_very_local} is best possible as can be seen 
by considering graphs that are a single cycle of length $r$. The cycle $o'$ in 
\autoref{local_is_very_local} is not only a cycle in $\expl(v,w)$ but also in $G$ by 
\autoref{new_rem}. 
\end{rem}

\begin{rem}
 The cycle $o'$ of \myref{local_is_very_local} contains an edge incident with $v$ whose other 
endvertex is not in $k$; that is, this other endvertex is equal to the vertex $w$ or else in a 
component of $\expl(v,w)-v-w$ different from $k$. 
\end{rem}

\begin{rem}\label{expl_explained}
 The notion of the explorer-neighbourhood is crucial in the proof of   \myref{recur_gd0_ONESTEP} 
and \myref{corner-lemma} below. This is explained in detail in \autoref{rem77} and 
\autoref{corner_{r/2}em} below. \end{rem}

\begin{rem}
Above we said the explorer-neighbourhood and the double-ball `almost lead' to the same notion of 
local 2-separator. This can be quantified as follows. If the punctured explorer-neighbourhood is 
connected, then so is the punctured double ball. If the punctured double ball of radius $r/2$ 
around 
two vertices of distance at most $d$ is connected, then the punctured explorer-neighbourhood of 
radius $(r/2)+d$ is connected. 
 
\end{rem}

\section{Intermezzo: Block-Cutvertex Graphs}\label{sec:blockcut}

The results of this section are not applied in the rest of the paper but they can be seen as a toy 
case for the main result. 
In this section we prove a generalisation of the block-cutvertex theorem allowing for 
$r$-local cutvertices, which generalise cutvertices (indeed, the $r$-local cutvertices for 
$r=\infty$ are precisely the cutvertices). See \autoref{sec2}for a defintion of $r$-local 
cutvertices and \autoref{sec:graph-deco} for a definition of graph-decompositions. 

It seems to us that the most natural generalisation of the block-cutvertex theorem to this 
context is the following.

\begin{thm}\label{block-cut}
 Given $r\in \Nbb\cup\{\infty\}$, every connected graph has a graph-decomposition of adhesion one 
and locality $r$
such that all its bags are $r$-locally 2-connected or single edges. 
\end{thm}

\begin{rem}
 The strengthening of \autoref{block-cut} with `bags are $r$-locally 2-connected' replaced by `bags 
are $r$-locally 2-connected subgraphs' is not true. An example is given in \autoref{fig:gluing1}.
\end{rem}

As a preparation for the proof of \autoref{block-cut}, we investigate the operation of locally 
cutting vertices, defined as follows. 

Given a parameter $r\geq 1$ and a graph $G$ with a vertex $v$, the graph 
obtained from $G$ by \emph{$r$-locally cutting} the vertex $v$ is defined as follows. 
Let $X$ be the set of connected components of the ball of radius $r$ around $v$ with $v$ removed; 
formally $X$ is the set of components of the graph $B_{r/2}(v)-v$. 
Define a new graph from $G$ by replacing the vertex $v$ by one new vertex for each element of the 
set $X$, where the vertex labelled with $x\in X$ inherits the incidences with those edges incident 
with $v$ that are incident with a vertex of the connected component $X$. 
We refer to the new vertices as the \emph{slices of $v$}. 
This completes the 
construction of the  $r$-local cutting of $G$. 

\begin{obs}\label{no_cut_vertex}
 Let $G'$ be obtained from $G$ by $r$-locally cutting a vertex $v$ into a set $X$ of new 
vertices. 
Then in the graph $G'$, no vertex $x\in X$ is an $r$-local cutvertex. 
\qed
\end{obs}

The next lemma says that  $r$-local cuttings commute.

\begin{lem}\label{commute}
Given a graph $G$ with vertices $v$ and $w$, first 
$r$-locally cutting $v$ and then $w$ results in the same graph as first locally cutting $w$ and 
then $v$. 
\end{lem}

\begin{proof}

Consider the graph $G'$ obtained from $G$ by $r$-locally cutting the vertex $v$.
We denote the  ball of radius $r/2$ around the vertex $w$ in the graph $G$ by $B_{r/2}(w)$, and by  
$B_{r/2}'(w)$ we denote the ball of radius $r/2$ around the vertex $w$ in the graph $G'$. 

In the graphs $G$ and $G'$, the vertex $w$ has the same neighbours. Indeed, if $v$ and $w$ are not 
adjacent, this is immediate. Otherwise $w$ is adjacent with a unique slice of $v$ in $G'$, and 
in the following we will suppress a bijection between the vertex $v$ and this particular slice 
of $v$ -- in order to simplify notation. 
With this notation at hand, we next prove the following.

\begin{sublem}\label{same_com}
 Two neighbours $x$ and $y$ of $w$ are in the same connected component of $B_{r/2}(w)-w$ if and 
only if 
they are in the same connected component of $B_{r/2}'(w)-w$.
\end{sublem}

\begin{proof}
If $x$ and $y$ are in the same connected component of $B_{r/2}'(w)-w$, they are joined by a path in 
that graph and this path is also is a path (or a walk) in the graph $B_{r/2}(w)-w$. 

Hence conversely assume that $x$ and $y$ are vertices of the same connected component of the 
ball $B_{r/2}(w)-w$. Let $P$ be a path between these two vertices in the graph $B_{r/2}(w)-w$. Then 
this path $P$ together with the vertex $w$ forms a cycle, which we denote by $o$. By \autoref{gen}, 
the cycle $o$ is generated by cycles of length at most $r$ in the graph $B_{r/2}(w)-w$. 

If one of these cycles does not include the vertex $v$, then it is also a cycle in the graph 
$B_{r/2}'(w)$. Otherwise, such a cycle is also a cycle completely contained with in the ball 
$B_{r/2}(v)$ 
around $v$ in $G$. In particular this cycle witnesses that the two neighbours on that cycle 
adjacent to $v$ are in the same connected component of  $B_{r/2}(v)-v$. Thus these two neighbours 
are 
neighbours of the same slice of the vertex $v$ in $G'$. Hence this cycle is also a cycle in $G'$ 
and hence in the ball $B_{r/2}'(w)$. To summarise, all those cycles of length at most $r$ that 
generate $o$ are cycles in $B_{r/2}'(w)$. In the ball $B_{r/2}'(w)$ they generate (the edge set of) 
$o$. So 
$o$ is an eulerian subgraph in $B_{r/2}'(w)$, and so a cycle as it cannot have a vertex of degree 
strictly more than two and it is connected. In particular the vertices $x$ and $y$ are in the same 
connected component of the punctured ball $B_{r/2}'(w)-w$. 
\end{proof}

It is a direct consequence of \autoref{same_com} that cutting locally commutes. 
\end{proof}

\begin{lem}\label{cut_all1}
 Let $G$ be a connected graph. Let $G'$ be obtained from $G$ by $r$-locally cutting all 
$r$-local cutvertices of $G$. Then $G'$ is $r$-locally 2-connected. 
\end{lem}

\begin{proof}
First we remark that the graph $G'$ is well-defined by \autoref{commute}. Let $v_1,v_2,...,v_n$ be 
an enumeration of the vertices of $G$. 
Here we stress that we include vertices in this enumeration that are not $r$-local cutvertices; and 
cutting them does not change the graph at all. 
We may assume by   \autoref{commute} 
that we obtain $G'$ 
from $G$ by first cutting $v_1$, then $v_2$, etc., so that in the final step we cut the 
vertex $v_n$. By \autoref{no_cut_vertex}, all slices of the vertex $v_n$ are not $r$-local 
cutvertices. As cutting locally commutes by  \autoref{commute}, we can argue the same for any other 
ordering of the vertices of $G$. Hence no vertex of the graph $G'$ is an $r$-local 
cutvertex. 
\end{proof}

\begin{proof}[Proof of \autoref{block-cut}.]
Let $r\in \Nbb\cup \{\infty\}$ be a parameter. 
 Let $G$ be a connected graph. We construct the graph $H$ from $G$ by $r$-locally cutting all 
$r$-local cutvertices of $G$.  By \autoref{commute} this is well-defined, 
and the graph $H$ is $r$-locally 2-connected by \autoref{cut_all1}.

Let $S$ be the set of $r$-local cutvertices of $G$. Let $B$ be the set of connected components of 
the graph $H$. 
We define a bipartite graph with bipartition $(B,S)$, where we add one edge between an $r$-local 
cutvertex $s$ of $G$ to a connected component $k$ of the graph $H$ for every slice of $s$ that 
is contained in $k$. The map associated to that edge map the singleton subgraph $s$ to its 
corresponding slice. We set $G_s=s$ and $G_b=b$ for $s\in S$ or $b\in B$, respectively. 

This defines a graph-decomposition of adhesion one and locality $r$ all of whose bags are 
$r$-locally 2-connected; compare \autoref{sec:graph-deco} for definitions. 
It is straightforward to check that the underlying graph of that graph-decomposition is the graph 
$G$. 
\end{proof}

\section{The existential statement of the local 2-separator theorem}\label{sec:ex}

In this section, we prove the lemmas necessary to deduce the existential statement of the local 
2-separator theorem; that is, the 
first sentence of \autoref{main_2sepr-intro}. 

 \addcontentsline{toc}{subsection}{Dfn: Local cutting}
\begin{dfn}[Local cutting]
 Given a graph $G$ with an $r$-local 2-separator $\{v_0,v_1\}$, the graph obtained from $G$ by 
\emph{$r$-locally cutting $\{v_0,v_1\}$} is defined as follows. 
Let $X$ be the set of connected components of the punctured explorer-neighbourhood 
$\expl(v_0,v_1)-v_0-v_1$. 
We now replace in the graph $G$ the vertices $v_0$ and $v_1$ each by one copy for every element of 
$X$. Here a copy of $v_i$ labelled by some $x\in X$ inherits an edge from $v_i$ if the other 
endvertex of that edge is a vertex of the component $x$.
We refer to the newly added vertices as the \emph{slices} of the vertices $v_1$ or $v_2$, 
respectively. 
We additionally add a weighted edge between any two 
slices for the same $x\in X$. Its weight 
is given by the minimum length of a path between $v_0$ and $v_1$ in the 
explorer-neighbourhood $\expl(v_0,v_1)$ with the component $x$ removed. It 
follows that all but one of these weights are 
always the same. We refer to these additional edges as \emph{torso edges}.
If the vertices $v_0$ and $v_1$ share an edge $e$ in $G$, 
we add a new connected component consisting of the edge $e$ and one edge in parallel to $e$. This 
other edge is a torso edge and its length is the minimum length of a path between $v_0$ and $v_1$ 
in 
the 
explorer-neighbourhood $\expl(v_0,v_1)$ minus $e$. Finally, we replace each torso edge by a path of 
the same 
length; we refer to such paths as \emph{torso-paths}\footnote{This technical step reduces 
technicalities elsewhere; indeed, the explorer-neighbourhood is not defined for weighted graphs, 
and doing so would lead to technicalities.}.
This completes the definition of local cutting, see \autoref{fig:gluing_new} for an example.
\end{dfn}

   \begin{figure} [htpb]   
\begin{center}
   	  \includegraphics[height=4cm]{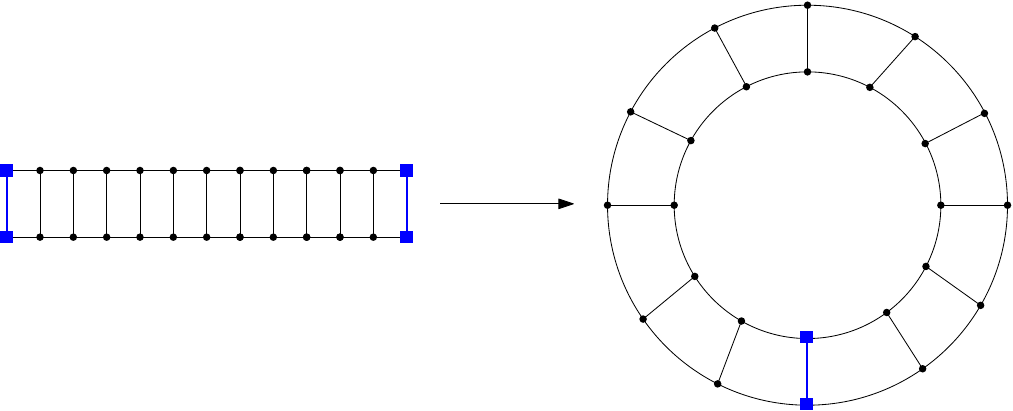}
   	  \caption{The graph on the left is obtained from the graph on the right by locally 
splitting at the local 2-separator given by the blue edge. This blue edge gets two 
copies on the left, one for each local component. In \autoref{sec:graph-deco} we shall 
investigate the inverse operation of local cutting.}\label{fig:gluing_new}
\end{center}
   \end{figure}

\begin{rem}
All edges incident with the vertices $v_0$ or $v_1$ are inherited by a unique 
slice except for possibly an edge between $v_0$ and $v_1$, which is in this artificial component 
of size two. If an edge between $v_0$ and $v_1$ in $G$ is a shortest path between these vertices, 
its length is the length of all torso edges. 
\end{rem}

\begin{lem}\label{cut_far}
 Slices of the same vertex have distance at least $r+1$.
 
 And slices for the same component have distance at most $r/2$ in a graph obtained from an 
$r$-locally $2$-connected graph by local cutting. 
\end{lem}
\begin{proof}
Firstly, suppose for a contradiction that there was a path $P$ of length at most $r$ joining two 
slices of a 
vertex $v$. Pick $P$ so that no interior vertex is a slice of $v$. Then all vertices on the path 
$P$ have distance at most $r/2$ from $v$ in the original graph. So the path $P$ projects to a 
closed walk completely contained in the ball $B_{\frac{r}{2}}(v)$. Hence all interior vertices of 
$P$ are in the same local component in any explorer-neighbourhood around $v$ and some other vertex 
$w$. 
So we get a contradiction to the assumption that $P$ joins two different slices of the vertex $v$. 

Secondly, let $\{a_1,a_2\}$ be an $r$-local $2$-separator of a graph $G$, and let $a_1'$ and $a_2'$
be slices for a local component $k$ in the graph $G'$ obtained from $G$ by local cutting.
As $G$ is $r$-locally 2-connected, by 
\myref{local_is_very_local} there is a cycle of length at most $r$ of $G$ through the vertices 
$a_1$ 
and $a_2$ that contains a vertex of the local component $k$. This cycle is a cycle of $G$ by 
\autoref{new_rem}. In the graph $G'$ there is a cycle $o'$ obtained from $o$ by possibly replacing 
one of its subpaths from $a_1$ to $a_2$ by a torso path. The cycle $o'$ has the same length as $o$ 
and contains the vertices $a_1'$ and $a_2'$, as it contains their neighbours from $k$. 
So $a_1'$ and $a_2'$ have distance at most $r/2$. 
\end{proof}

\begin{lem}\label{inverse_sum_cut}
 Let $G'$ be obtained from a graph $G$ by $r$-locally cutting a local 2-separator $\{v,w\}$.
 Then $G$ can be obtained from $G'$ by $r$-local sums.
\end{lem}

\begin{proof}
 The family of graphs for the local sum consists of copies of the graph $G'$, one copy for each 
component of the punctured explorer-neighbourhood $\expl(v,w)-v-w$, together with the artificial 
component of size two if $vw$ is an edge of $G$. We move from $G'$ to a weighted graph by 
replacing the torso paths by the torso edges of the cutting. 
Now we take an $r$-local $2$-sum where the gluing edges are the torso 
edges. 

It follows directly from the definitions of local cutting and local sums that the graph $G$ is 
equal to the graph obtained from $G'$ by applying the local 2-sum as described above. This 
local sum is $r$-local by \autoref{cut_far}. 
\end{proof}

\begin{lem}\label{loc2con_pres}
 Let $G'$ be a graph obtained from an $r$-locally 2-connected graph $G$ by $r$-locally cutting a 
local 2-separator. Then the graph $G'$ is $r$-locally 2-connected.
\end{lem}

\begin{proof}
By \myref{local_is_very_local}, every connected component of the graph $G'$ contains a cycle of 
length at 
most $r$. So it remains to 
show that there are no $r$-local cutvertices. 
Let $v$ be an arbitrary vertex of the graph $G'$.
We distinguish two cases.

{\bf Case 1:} the vertex $v$ is a slice.
We denote the local $2$-separator of $G$ at which we cut by $\{a,b\}$. We may assume, and we do 
assume, that the vertex $v$ is a slice of the vertex $a$. 
Let $X$ denote the set of components of the punctured explorer-neighbourhood $\expl(a,b)-a-b$. 
Recall that the ball $B_{r/2}(a)$ of radius $r/2$ around $a$ in the graph $G$ is a naturally 
embedded subgraph of the 
explorer-neighbourhood $\expl(a,b)$ (and we shall suppress this natural embedding).
For each component $x\in X$, we let $H_x$ be the intersection of the punctured ball $B_{r/2}(a)-a$ 
with the component $x$. 
Note that $b$ has distance at most $r/2$ 
from $a$ in $G$ by \myref{local_is_very_local}. 
The punctured ball $B_{r/2}(a)-a$ is obtained by taking 
the union of the graphs $H_x$ and adding the vertex $b$ together with its incident edges.
As this punctured ball is connected by assumption, all graphs $H_x$ must have the vertex $b$ in 
their neighbourhood and all graphs $H_x+b$ must be connected. 
The punctured ball $B_{r/2}(v)-v$ around $v$ in $G'$ is $H_y+b$, where $y$ is the component of 
$\expl(a,b)-a-b$ that belongs to the slice $v$. So $B_{r/2}(v)-v$ is connected. 
Thus the vertex $v$ is not an $r$-local cutvertex. 
This completes Case 1.

{\bf Case 2:} the vertex $v$ is not a slice. 
Then the vertex $v$ is a vertex of the graph $G$.  

Suppose for a contradiction that the punctured ball $B_{r/2}(v)-v$ around $v$ of radius $r/2$ in 
the 
graph $G'$ is 
disconnected.
Let $w_1'$ and $w_2'$ be two arbitrary neighbours of $v$ in $G'$ in different components of that 
punctured ball. 
Let $w_1$ and $w_2$ be the vertices of the graph $G$ from which the vertices $w_1'$ and $w_2'$ are 
slices of or that are equal to them, respectively. Then the vertices $w_1$ and $w_2$ are adjacent 
to the vertex $v$ in the 
graph $G$ by the definition of local cutting. As the punctured ball $B_{r/2}(v)-v$ of radius $r/2$ 
around the 
vertex $v$ in the 
graph $G$ is connected by assumption, there is a path $P$ within that punctured ball from $w_1$ to 
$w_2$. 
This path together with the vertex $v$ is a cycle $o$ within that ball. So by \autoref{gen} this 
cycle is generated by cycles within that ball of length at most $r$. 

Let $W'$ be the set of neighbours of the vertex $v$ in the graph $G'$ in the component of the 
punctured ball containing the vertex $w_1'$. Let $W$ be the set of vertices of the graph $G$ that 
are equal to vertices in $W'$ or that have slices in the set $W'$. 
By $E(W)$ we denote the set of edges in the graph $G$ from $v$ to a vertex in $W$.

By construction the cycle $o$ 
contains precisely one edge from the set $E(W)$. Hence there must be a cycle $\hat o$ of $G$ from 
the 
generating set that contains an odd number of edges from $E(W)$. 
As $\hat o$ has maximum degree two, it contains precisely one edge from the set $E(W)$.

We denote the local $2$-separator of $G$ at which we locally cut by $\{a,b\}$.

{\bf Case 2A:} the cycle $\hat o$ does not contain any of the vertices $a$ or $b$.
Then $\hat o-v$ is a path in the graph $G'$ from a vertex of $W'$ to a neighbour of the vertex 
$v$ outside $W'$. This is a contradiction to the assumption that the punctured ball is 
disconnected. This 
completes this case.

{\bf Case 2B:} the cycle $\hat o$ contains one of the vertices $a$ or $b$.
If $\hat o$ was a cycle of the graph $G'$, then we would get the desired contradiction as the path 
$\hat o-v$ would join two vertices in different components of the punctured ball around $v$ in $G'$. 
So assume that this is not the case. Then $\hat o$ contains both vertices $a$ and $b$. Let $Q$ be 
the $a$-$b$-subpath of $\hat o$ containing $v$. 
Then $Q$ plus a torso edge between two slices of $a$ and $b$ is a cycle of $G'$ whose length is no 
longer than the length of $\hat o$. Denote this cycle by $o'$. Then $o'-v$ joins two vertices in 
different components of the punctured ball around $v$ in $G'$, a contradiction. 
This completes Case 2, and hence the whole proof. 
\end{proof}

\begin{rem}
 In \autoref{ex2} we give an alternative proof of the first sentence of \autoref{main_2sepr-intro} 
that 
only relies on lemmas of the paper proved up to this point. We encourage the reader to look at this 
proof next.  
\end{rem}

\section{Properties of local 2-separators}\label{props}

In this section we prove some lemmas that are used in our proof of \autoref{structure_torso} and 
\autoref{thm:main_intro}. 

A \emph{cut} is the set of edges between a bipartition of the vertex set. The bipartition classes 
are referred to as the \emph{sides} of the cut. 

\begin{lem}\label{traverse_standard2}
 Let $Y$ be a cut in a graph $G$. 
 Then the endvertices of a path $P$ are on the same side of $Y$ if and only if $P$ intersects $Y$ 
evenly.
\end{lem}

\begin{proof}[Proof:]
 by induction on the length of the path $P$. 
\end{proof}

 \addcontentsline{toc}{subsection}{Dfn: Crossing}
\begin{dfn}[Crossing]
Given an $r$-local 
2-separator $\{v,w\}$ and a pair of vertices $\{a,b\}$ of the 
explorer-neighbourhood $\expl(v,w)$, we say that $\{a,b\}$ \emph{pre-crosses} $\{v,w\}$ if 
$a$ and $b$ are in different components of the punctured 
explorer-neighbourhood $\expl(v,w)-v-w$.
And $\{a,b\}$ \emph{crosses} the $r$-local 
2-separator $\{v,w\}$ if it pre-crosses it and there is a cycle of length at most $r$ in 
$\expl(v,w)$ through 
$a$ 
and $b$ in the explorer-neighbourhood; note that this cycle contains $v$ and $w$\footnote{We 
remark that this definition is a little subtle, as there may well be vertices $a$ and $b$ of $G$ 
that lie on a common cycle of $G$ of length at most $r$ and that are in different components of }. 

We say that a pair $\{a,b\}$ of (distinct) vertices of $G$ \emph{crosses} a local 2-separator 
$\{v,w\}$ of $G$ if there 
exist copies $a'$ and $b'$ of $a$ and $b$ in the explorer-neighbourhood $\expl(v,w)$, respectively, 
so that $\{a',b'\}$ crosses $\{v,w\}$. 
\end{dfn}

\begin{rem}
If $\{a,b\}$ is a local separator in an $r$-locally $2$-connected graph, then the existence of a 
cycle $o$ of length at most $r$ 
through $a$ and $b$ is guaranteed by  \myref{local_is_very_local}. Hence 
`crossing' essentially means `pre-crossing' plus the crossing vertices are `near' to the local 
separator. 
Phrasing being `near' in terms of this cycle seems particularly natural in view of 
\myref{alt_exist} and  \autoref{cross_sym} below. 
\end{rem}

 \addcontentsline{toc}{subsection}{Dfn: Alternating cycle}
\begin{dfn}[Alternating cycle]
 Given two disjoint sets $A_1$ and $A_2$, we say that a cyclic ordering \emph{alternates} between 
$A_1$ and $A_2$ if it has even length and each element of the cyclic ordering in $A_i$ has its two 
neighbours in $A_{i+1}$ (for $i\in \Fbb_2$). 

A \emph{pre-alternating cycle} is a cycle $o$ together with two local 
2-separators $\{a_1,a_2\}$ and $\{b_1,b_2\}$ such that the order in which these four vertices 
appear on the cycle $o$ alternates between the two local separators (i.e., it is $a_1b_1a_2b_2$ or 
its reverse $a_1b_2a_2b_1$). An \emph{alternating cycle} is a pre-alternating cycle $o$ such that 
the two neighbours of $a_i$ on $o$ are in different components of $\expl(a_1,a_2)-a_1-a_2$ for 
$i=1,2$, and analoguously the two neighbours of $b_i$ on $o$ are in different components of 
$\expl(b_1,b_2)-b_1-b_2$ for 
$i=1,2$. 
Below sometimes it will be more convenient to refer to this situation by saying that the 
cycle $o$ \emph{alternates between} the local 2-separators $\{a_1,a_2\}$ and $\{b_1,b_2\}$, see 
\autoref{fig:alt_cycle}.
\end{dfn}

   \begin{figure} [htpb]   
\begin{center}
   	  \includegraphics[height=4cm]{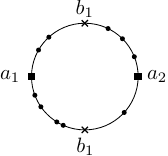}
   	  \caption{An alternating cycle. The vertices $a_1$ and $a_2$ of the first local 
2-separator are indicated by boxes, the vertices $b_1$ and $b_2$ of the second 
local 2-separator are indicated by crosses. The cyclic order of the cycle 
induced on these four vertices alternates between the two local separators.}\label{fig:alt_cycle}
\end{center}
   \end{figure}

   The `nearness' condition in the definition of `crossing' is equivalent to the following stronger 
property.

\begin{lem}[Alternating Cycle Lemma]\label{alt_exist}
Let $\{a_1,a_2\}$ and $\{b_1,b_2\}$ be $r$-local 2-separators in an $r$-locally $2$-connected 
graph $G$. The following are equivalent:
\begin{enumerate}
 \item $\{a_1,a_2\}$ crosses $\{b_1,b_2\}$;
 \item there is a cycle of $G$ of length at most $r$ alternating between $\{a_1,a_2\}$ and 
$\{b_1,b_2\}$;
 \item every cycle of length at most $r$ of $G$ containing $a_1$ and $a_2$ alternates between 
$\{a_1,a_2\}$ and $\{b_1,b_2\}$;
\item every cycle of length at most $r$ of $G$ containing $a_1$ and $a_2$ pre-alternates between 
$\{a_1,a_2\}$ and $\{b_1,b_2\}$ and the neighbours of $b_1$ on $o$ are in different 
components of $\expl(b_1,b_2)-b_1-b_2$. 
\end{enumerate}
\end{lem}

\begin{proof}
 To see that (3) implies (2) note that by $r$-locally $2$-connectivity there is a cycle of length at 
most $r$ containing $a_1$ and $a_2$ by \myref{local_is_very_local}; and it is a cycle of $G$ by 
\autoref{new_rem}.

To see that (2) implies (1), let $o$ be a cycle alternating between $\{a_1,a_2\}$ and $\{b_1,b_2\}$.
By \autoref{unique_copy_extended}, $o$ is a cycle of $\expl(b_1,b_2)$ and all its vertices have 
unique copies in $\expl(b_1,b_2)$. For simplicity, we suppress a map between the vertices $a_i$ and 
their unique copies in  $\expl(b_1,b_2)$. As $o$ is alternating, the vertices $a_1$ and $a_2$ are 
in different components of the punctured explorer-neighbourhood $\expl(b_1,b_2)-b_1-b_2$.

Next we show that (1) implies (4). 
For this assume that the vertices $a_1$ and $a_2$ have copies $a_1'$ and $a_2'$, respectively, in 
different components of the punctured explorer-neighbourhood $\expl(b_1,b_2)-b_1-b_2$ such that 
there is a cycle $o$ of $\expl(b_1,b_2)$ of length at most $r$ containing $a_1'$ and $a_2'$. 
So the cycle $o$ has to intersect the 2-separator $\{b_1,b_2\}$ of $\expl(b_1,b_2)$ in both its 
vertices. Thus by \autoref{unique_copy_extended}, all vertices on $o$ have unique copies in 
$\expl(b_1,b_2)$. For simplicity, we suppress a map between the vertex $a_i$ and its unique copy 
$a_i'$ in $\expl(b_1,b_2)$. We have shown that in $G$ there is an $a_1$-$a_2$-path $P$ of length at 
most $r/2$ containing one of the vertices $b_i$, say $b_1$. 

Now let $u$ be an arbitrary cycle of length at most $r$ containing $a_1$ and $a_2$. Let $Q$ be a 
subpath of $G$ of that cycle from $a_1$ to $a_2$ of length at most $r/2$. The closed walk $PQ$ has 
length at most $r$. As it contains the vertex $b_1$, by \autoref{new_rem} it is a closed walk in 
$\expl(b_1,b_2)$. Hence the path $Q$ traverses the separator $\{b_1,b_2\}$. So it contains a vertex 
$b_i$. So by \autoref{new_rem}, the cycle $u$ is a cycle of $\expl(b_1,b_2)$.
It follows that the cycle $u$ pre-alternates between the local 2-separators $\{a_1,a_2\}$ and 
$\{b_1,b_2\}$. As $a_1$ and $a_2$ are in different components of $\expl(b_1,b_2)-b_1-b_2$, this 
completes the proof that (1) implies (4).

Finally, we show how (4) implies (3). 
Let $u$ be a cycle length at most $r$ of $G$ containing $a_1$ and $a_2$. By (4) it pre-alternates 
between 
$\{a_1,a_2\}$ and $\{b_1,b_2\}$ and the neighbours of $b_1$ on $o$ are in different 
components of $\expl(b_1,b_2)-b_1-b_2$. 
By \autoref{unique_copy_extended}, the vertices $b_1$ and $b_2$ have unique copies in 
$\expl(a_1,a_2)$, and we suppress a map between the vertex $b_i$ and its unique copy 
 in $\expl(a_1,a_2)$. It suffices to show that the vertices $b_1$ and $b_2$ are in different 
components of $\expl(a_1,a_2)-a_1-a_2$. 

Suppose for a contradiction that they are in the same component. By \myref{local_is_very_local}, 
there is a cycle $c$ of length at 
most $r$ in $\expl(a_1,a_2)$ containing $a_1$ and $a_2$ such that one of its $a_1$-$a_2$-subpaths 
does not contain a vertex $b_i$. By \autoref{new_rem}, $c$ is a cycle of $G$. 
Let $R$ be an 
$a_1$-$a_2$-subpath of $c$ of length at most $r/2$. Then $PR$ is a closed walk of $G$ of length at 
most $r$. As it contains the vertex $b_1$, it is a cycle of $\expl(b_1,b_2)$. So $R$ contains a 
vertex $b_i$. So by \autoref{new_rem}, $c$ is a cycle of $\expl(b_1,b_2)$.
As the vertices $a_1$ and $a_2$ are in different components of 
$\expl(b_1,b_2)-b_1-b_2$, every subpath of $c$ has to contain a vertex $b_i$, a contradiction to 
the choice of $c$. 
Hence the vertices $b_1$ and $b_2$ are in different components of 
$\expl(a_1,a_2)-a_1-a_2$. So every cycle $u$ alternates between $\{a_1,a_2\}$ and $\{b_1,b_2\}$. 
\end{proof}

The next lemma essentially says that crossing is a symmetric relation on $r$-local 
2-separators. 

\begin{cor}\label{cross_sym}
Let $G$ be an $r$-locally 2-connected graph with two $r$-local 2-separators $\{a_1,a_2\}$ and
$\{b_1,b_2\}$.
If $\{a_1,a_2\}$ crosses $\{b_1,b_2\}$, then $\{b_1,b_2\}$ crosses $\{a_1,a_2\}$.

Furthermore the punctured explorer-neighbourhood $\expl(a_1,a_2)-a_1-a_2$ has precisely two 
components. 
\end{cor}

\begin{proof}
For the first part assume that $\{a_1,a_2\}$ crosses $\{b_1,b_2\}$. 
This is equivalent to the symmetric condition (2) in \myref{alt_exist}, so it follows that 
$\{b_1,b_2\}$ crosses $\{a_1,a_2\}$.

To see the `Furthermore'-part, by condition (3) of  \myref{alt_exist}
every cycle 
$c$ of length at 
most $r$ in $\expl(a_1,a_2)$ containing $a_1$ and $a_2$ has to contain a vertex $b_i$ on each of 
its $a_1$-$a_2$-subpaths. By the contrapositive of \myref{local_is_very_local}, we conclude that 
every component of $\expl(a_1,a_2)-a_1-a_2$ contains a vertex $b_i$. 
So $\expl(a_1,a_2)-a_1-a_2$ has at most two components, and it has exactly two as $\{a_1,a_2\}$ is 
a local separator. 
\end{proof}

\begin{rem}\label{corner_{r/2}em}
 A key-feature of separators in graphs is the `Corner Property\footnote{In modern terms, it 
just says that the connectivity function is submodular.}\rq.
In the classic version for 2-separators, the Corner-Lemma says that if two 2-separators 
$\{a_1,a_2\}$ and $\{b_1,b_2\}$ cross, then $\{a_1,b_1\}$ is a 2-separator -- under certain 
non-triviality conditions. We shall prove that this property also holds for our local 2-separators.
This is in fact a central lemma of the paper 
and the notion of `explorer-neighbourhood' is key to this lemma.

Intuitively speaking, the reason why this lemma is true is the following. 
As the Corner-Lemma is 
true for separators in the classical version, the only reason why a local version could 
break is 
essentially if one of the involved vertices, say $a_1$,  
would explore a new path around the local separator $\{b_1,b_2\}$ to the other component of the 
punctured explorer-neighbourhood $\expl(b_1,b_2)-b_1-b_2$ that was not known to $b_1$ or $b_2$. 
If we used `double balls' instead of our local notion of `explorer-neighbourhoods', this could well 
happen, see \autoref{fig:no-corner} below for an example. The intuition now is that $a_1$ may well 
`explore' a  new path to the other local component but when the explorers compare their 
maps, they have given the things different names and so the explorers do not realise that between 
them they know a path around. Hence they believe that the corner $\{a_1,b_1\}$ is separating. That 
is how we think about locally separating: the explorers cannot prove that there is a way round 
with their local information. 

This is somewhat similar to the following situation. Imagine you are running on a graph and at 
any point in time you can only see your neighbours. If the graph is a cycle, you cannot tell its 
length -- and you even cannot distinguish it from the 2-way-infinite path $\Zbb$. 
\end{rem}

Now we start setting up some notation for \autoref{corner-lemma} below. 

 \addcontentsline{toc}{subsection}{Dfn: Corner Setting}
\begin{dfn}[Corner-Setting]
 Given an $r$-local separator $\{a_1,a_2\}$ crossing an $r$-local 2-separator $\{b_1,b_2\}$,
by \autoref{cross_sym} and \myref{alt_exist}, $\{b_1,b_2\}$ crosses $\{a_1,a_2\}$ and there is a 
cycle $o$ of 
length at most $r$ alternating between these two local separators.
A \emph{person of type one} is a vertex $x$ of $V(o)-a_1-a_2-b_1-b_2$ that has a copy in the same 
component of $\expl(a_1,a_2)-a_1-a_2$ as $b_1$ and a copy in the same 
component of $\expl(b_1,b_2)-b_1-b_2$ as $a_1$.
A \emph{person of type two} is a neighbour $x$ of $b_1$ outside the cycle 
$o$ that has a copy in the same component of $\expl(b_1,b_2)-b_1-b_2$ as $a_1$.
A \emph{person of type three} is a neighbour $x$ of $a_1$ outside the cycle 
$o$ that has a copy in the same component of $\expl(a_1,a_2)-a_1-a_2$ as $b_1$.
A \emph{person} is a person of type one, two or three. 
We say that a person \emph{lives in the corner $\{a_1,b_1\}$} if there exists a person.
\end{dfn}

\begin{rem}
 A person of type one has unique copies in $\expl(a_1,a_2)$, 
$\expl(b_1,b_2)$ and $\expl(a_1,b_1)$ by \autoref{unique_copy_extended}.
For a person $x$ of type two, the edge $xb_1$ has unique copies in
$\expl(b_1,b_2)$ and 
$\expl(a_1,b_1)$ by \autoref{unique_copy}. Hence we define the copy of a person of type two in 
these neighbourhoods to be the unique endvertex of that edge that is a copy of $x$. 
Hence  to simplify notation, below we suppress a bijection between a person and its unique copies 
in the explorer-neighbourhoods where  copies are unique. 
\end{rem}

\begin{lem}[Corner Lemma]\label{corner-lemma}
Assume $G$ is $r$-locally 2-connected and assume the corner-setting. 
Assume a person $x$ lives in the corner $\{a_1,b_1\}$, then $\{a_1,b_1\}$ 
is an $r$-local 2-separator. 

Moreover, $x$ is in a different component of $\expl(a_1,b_1)-a_1-b_1$ than (copies of) $a_2$ and 
$b_2$.
\end{lem}

\begin{proof}
 By \autoref{unique_copy_extended}, the vertices $a_1$, $a_2$, $b_1$ and $b_2$ have unique copies 
in $\expl(a_1,b_1)$; hence for this proof we suppress a bijection between these vertices and their 
copies in $\expl(a_1,b_1)$.
We start by showing the following. 

\begin{sublem}\label{sepr_inexpl1}
 The vertices $b_1$ and $b_2$ are in different components of the graph $\expl(a_1,b_1)-a_1-a_2$.
\end{sublem}

\begin{proof}
Suppose not for a contradiction. Then there is a path $P$ of the graph 
$\expl(a_1,b_1)-a_1-a_2$ from $b_1$ to $b_2$. 

Let $W$ be the neighbourhood of the set $\{a_1,a_2\}$ in the 
punctured explorer-neighbourhood $\expl(a_1,a_2)-a_1-a_2$ in the component containing $b_1$. 
By \autoref{unique_copy} neighbours of $a_1$ or $a_2$ have unique copies in the 
explorer-neighbourhood $\expl(a_1,a_2)$; hence there is a bijection between the neighbours of $a_1$ 
and $a_2$ in $G$ and the explorer-neighbourhood. To simplify notation we suppress this map in 
our notation. And we will simply consider $W$ as a vertex set of the graph $G$, as well. 
Let $E(W)$ be the set of edges of the graph $G$ from $\{a_1,a_2\}$ to $W$.

Recall that by the corner-setting, there is a cycle of length at most $r$ alternating between the 
local 
separators $\{a_1,a_2\}$ and $\{b_1,b_2\}$. Thus it has a subpath $Q$ from $b_1$ to $b_2$ of length 
at 
most $r/2$; this subpath contains precisely one of the vertices $a_1$ and $a_2$. 
This alternating cycle is also a cycle of the explorer-neighbourhood $\expl(a_1,a_2)$ and the path 
$Q$ has to intersect the set $E(W)$ oddly\footnote{In fact it intersects this set just once but we 
will not need that strengthening.} as it connects vertices in different components. 
Vertices of $o$ have unique copies in $\expl(a_1,b_1)$ by \autoref{unique_copy_extended}. We 
suppress a bijection between vertices of $o$ and their unique copies in $\expl(a_1,b_1)$.
So in $\expl(a_1,b_1)$, $Q$ is a path from $b_1$ to $b_2$. 

Let $u$ be the closed walk of the explorer-neighbourhood $\expl(a_1,b_1)$
obtained by concatenating $P$ and $Q$. 
The path $P$ considered as a walk of the graph $G$ contains no vertex $a_i$ and thus does 
not intersect the edge set $E(W)$. Thus the closed walk $u$, considered as an edge set of the 
graph 
$G$ intersects the edge set $E(W)$ in 
an odd number of edges. 

By \autoref{cycle_gen} the closed walk $u$ is generated by cycles of $G$ included in the balls 
$B_{r/2}(a_1)$ 
and $B_{r/2}(b_1)$. These cycles are in turn by \autoref{gen} generated by cycles of length at most 
$r$ included in these 
balls. 
Hence one of the generating cycles has to intersect the edge set $E(W)$ oddly. Call such a cycle 
$u'$. So the cycle $u'$ contains the vertex $a_1$ or $a_2$. As it has bounded length, it is a cycle 
of the explorer-neighbourhood $\expl(a_1,a_2)$. This is a contradiction as in the 
explorer-neighbourhood $\expl(a_1,a_2)$ the cycle $u'$ and the cut $E(W)$ cannot intersect oddly. 
Thus the vertices $b_1$ and $b_2$ must be in different components of the graph 
$\expl(a_1,b_1)-a_1-a_2$.
\end{proof}

By exchanging the roles of the `$a_i$' and `$b_i$' in \autoref{sepr_inexpl1} one obtains the 
following.

\begin{sublem}\label{sepr_inexpl2}
  The vertices $a_1$ and $a_2$ are in different components of the graph $\expl(a_1,b_1)-b_1-b_2$.
\end{sublem}

\begin{proof}
The proof is analogous to that of \autoref{sepr_inexpl1}. 
\end{proof}

By $C(a,i)$ we denote the component of $\expl(a_1,b_1)-a_1-a_2$ containing the vertex $b_i$.
By $C(b,i)$ we denote the component of $\expl(a_1,b_1)-b_1-b_2$ containing the vertex $a_i$.

By assumption there is vertex $x$ of $G$ that is a person living in the corner $\{a_1,b_1\}$.
If $x$ is a person of type one, then it is contained in both components $C(a,1)$ and $C(b,1)$, as 
the vertices $b_1$ and $a_1$ are in the respective components.
If $x$ is a person of type two, then it is contained in the component $C(a,1)$ as $b_1\in C(a,1)$.
Next we verify that it is in $C(b,1)$. By definition $a_1\in C(b,1)$. As $x$ is in the sam 
component of $\expl(b_1,b_2)-b_1-b_2$ as $a_1$, pick a path joining them in there and add the edges 
$b_1x$ and the subpath $b_1oa_1$ of $o$ that avoids $b_2$. This forms a cycle. 
By \autoref{cycle_gen} and \autoref{gen}, this cycle is generated by cycles of length at most 
$r$. From these cycles take those that contain the vertex $b_1$. The sum of these cycles with the 
vertex $b_1$ removed gives a closed walk in $B_{r/2}(b_1)$ from $x$ to the neighbour of $b_1$ on 
$b_1oa_1$ (which is equal to $a_1$ or a person of type one). This walk witnesses that $x\in 
C(b,1)$. 
Like for persons of type two, we show for persons of type three that they are contained in 
$C(a,1)$ and $C(b,1)$. 
To summarise, in every case the intersection of 
 $C(a,1)$ and $C(b,1)$ contains the person $x$.
 
In particular, the intersection of $C(a,1)$ and $C(b,1)$ is nonempty and includes a component of 
$\expl(a_1,b_1)-a_1-a_2-b_1-b_2$. Denote such a component containing the person $x$ by 
$k$. As the component $k$ is included in $C(a,1)$, it does not contain any neighbour of the vertex 
$b_2$ by \autoref{sepr_inexpl1}. Similarly, $k$ does not contain any neighbour of the vertex $a_2$ 
by \autoref{sepr_inexpl2}. Hence $k$ is also a component of $\expl(a_1,b_1)-a_1-b_1$. As $k$ does 
not contain the vertex $a_2$, the punctured explorer-neighbourhood $\expl(a_1,b_1)-a_1-b_1$ is 
disconnected. 
Thus $\{a_1,b_1\}$ is an $r$-local 2-separator. 

The `Moreover'-part is clear by construction. 
\end{proof}

\begin{rem}
 In the next paragraph we define \lq contacts\rq. This definition is technical and results in quite 
a few technicalities later on. In contrast to this, the intuition behind contacts is rather 
simple:
given a graph $G'$ obtained from a graph $G$ by locally cutting at $\{a_1,a_2\}$. 
We would like to move back and forth between local separators of $G$ and $G'$.
Specifically, if $\{b_1,b_2\}$ is disjoint from $\{a_1,a_2\}$, then $\{b_1,b_2\}$ is a local 
separator of $G$ if and only if it is a local separator of $G'$; compare \myref{recur_gd0_ONESTEP} 
and \myref{lifting-lem} below. If the sets $\{b_1,b_2\}$ and $\{a_1,a_2\}$ are not disjoint, 
we still want to move back and forth between local separators, and contacts are our way to 
formalise it; they associate to every vertex $b_i$ of $G$ a canonical vertex of $G'$ relative to 
$\{b_1,b_2\}$. (Imagine you phone a \lq helpline\rq\ and they redirect your call to the \lq  next 
contact\rq. Don't worry: unlike for many helplines, here this happens only once.) The proofs below 
would simplify a lot if one (carelessly) assumed that $\{b_1,b_2\}$ and $\{a_1,a_2\}$ were disjoint.
\end{rem}

 \addcontentsline{toc}{subsection}{Dfn: Contacts}
\begin{dfn}[Contacts]
 Given a graph $G'$ obtained from $G$ by locally cutting a local 2-separator, by definition there 
is a bijection
between the edges of $G$ and the edges of $G'$ that are not on torso paths. To simplify 
notation, we 
 suppress this bijection from our notation. 
Let $b$ be a vertex of $G$ and let $b'$ be a vertex of $G'$ that is equal to $b$ or a slice 
thereof.  Let $e$ be an edge of $G$ that is incident with the vertex $b$. Then the edge $e$ is 
incident 
with the vertex $b'$ in $G'$ or else the vertex $b'$ must be a slice, and thus is incident with a 
unique edge on a torso path. 
The \emph{contact} of the edge $e$ at the vertex $b'$ is the edge $e$ itself if $e$ is incident 
with $b'$ in $G'$ or else the contact is the unique edge on a torso path incident with $b'$. 
\end{dfn}

Given a local 2-separator $\{b_1,b_2\}$ in a graph $G$ and two edges $e$ and $f$ each with 
one endvertex in $\{b_1,b_2\}$, we say that $e$ 
and $f$ are \emph{separated} by $\{b_1,b_2\}$ if 
the edges $e$ and $f$ have endvertices in different components of the punctured 
explorer-neighbourhood $\expl(b_1,b_2)-b_1-b_2$.

For a vertex $x'$ of $G'$ that is not an interior vertex of a torso-path, there is a unique vertex 
of $G$ that is equal to $x'$ or such that $x'$ 
is a slice of that vertex. We denote this vertex by $x$.

\begin{lem}[Projection Lemma]\label{recur_gd0_ONESTEP}
Assume $G$ is $r$-locally 2-connected. 
For any $r$-local $2$-separator 
$\{b_1',b_2'\}$ of $G'$ such that the $b_i'$ are not interior vertices of torso-paths, the set 
$\{b_1,b_2\}$ is an $r$-local $2$-separator of $G$. 

More specifically, edges $e$ and $f$ each with 
one endvertex in $\{b_1,b_2\}$ are separated by $\{b_1,b_2\}$ in $G$ if 
their contacts are separated by $\{b_1',b_2'\}$ in $G'$.
\end{lem}

\begin{proof}
In this proof we will distinguish between the vertices of $G$ and $G'$ by adding a dash to the 
vertices of the graph $G'$; for example we write $b_1'$ when we consider $b_1$ as a vertex of $G'$ 
and $b_1$ when we consider it as a vertex of the graph $G$ (whenever $b_1'$ is not an interior 
vertex of a torso path). We prove the \lq More specifically\rq-part, as this suffices.

Let $e$ and $f$ be edges incident with 
precisely one of $b_1$ or $b_2$ such that their contacts are separated by $\{b_1',b_2'\}$ in $G'$.
\begin{sublem}\label{at_most_one_torso}
 There is at most one edge on a torso path incident with vertices of  $\{b_1',b_2'\}$. 
\end{sublem}

\begin{proof}
Let $\{a_1,a_2\}$ be a local 2-separator of $G$ such that $G'$ is obtained from $G$ by locally 
cutting at $\{a_1,a_2\}$. 
Suppose for a contradiction that there are two edges on torso paths incident with vertices of 
$\{b_1',b_2'\}$. As each vertex is incident with at most one torso edge, both $b_1'$ and $b_2'$ 
must be slices. 

 As $\{b_1',b_2'\}$ is an $r$-local 2-separator, the vertices $b_1'$ and $b_2'$ have distance at 
most $r/2$. 
 So $b_1'$ and $b_2'$ cannot be slices of the same vertex by \autoref{cut_far}. 
 By symmetry assume that $b_1'$ is a slice of $a_1$ and $b_2'$ is a slice of $a_2$.
 As there are at least two torso edges, $b_1'$ and $b_2'$ must be slices for different components  
of $\expl(a_1,a_2)-a_1-a_2$. 
As $\{a_1,a_2\}$ is an $r$-local $2$-separator, the vertices $a_1$ and $a_2$ have distance at most 
$r/2$. It follows from the definition of local cutting that slices of $a_1$ and $a_2$ have the same 
distance as $a_1$ and $a_2$, so this distance is upper-bounded by $r/2$. 
 
Let $a_1'$ be the slice of the vertex $a_1$ for the same component as the slice $b_2'$ of $a_2$. 
So the distance from $a_1'$ to $b_2'$ is at most $r/2$. So the distance between the distinct slices 
$a_1'$ and $b_1'$ of $a_1$ is at most $r$. This is a contradiction to \autoref{cut_far}. 
Hence there is at most one torso edge incident with vertices of  $\{b_1',b_2'\}$. 
\end{proof}

Let $k'$ be the component of the punctured explorer-neighbourhood $\expl(b_1',b_2')-b_1'-b_2'$ in 
$G'$ that contains an endvertex of the contact for $e$\footnote{This is well-defined as the 
contact for $e$ is an edge that has exactly one endvertex in $\{b_1',b_2'\}$, and the other 
endvertex defines the component $k'$ uniquely.}.  
Let $W$ be the set of edges of $G'$ with one endvertex in $\{b_1',b_2'\}$ and the other endvertex 
in $k'$. 
By \autoref{at_most_one_torso} and by 
exchanging the roles of the edges $e$ and $f$ if necessary, we may assume, and we do assume, that 
no edge on a torso path incident with a vertex of $\{b_1',b_2'\}$ has its other endvertex in 
the component 
$k'$. Hence the edge set $W$ is also an edge set of the graph $G$. 
By \autoref{unique_copy}, edges of $W$ have unique copies in $\expl(b_1,b_2)$. For simplicity we 
suppress a bijection between $W$ and its uniquely defined copy in $\expl(b_1,b_2)$.

\begin{sublem}\label{path_exists1}
 There is a path $Q$ from $b_1$ to $b_2$ contained in $\expl(b_1,b_2)$ that contains an 
even number of edges from $W$. 
\end{sublem}

\begin{proof}
As $G'$ is $r$-locally 2-connected by \autoref{loc2con_pres}, by \myref{local_is_very_local} 
there is a cycle $o'$ of length at most $r$ included in 
$\expl(b_1',b_2')$ containing the vertices $b_1'$ and $b_2'$. By 
\autoref{cut_far}, the cycle $o'$ can contain edges of at most one torso path (note that if a 
cycle 
contains edges from a torso path, it must include the whole torso path). Hence there is a 
path $Q'$ 
from $b_1'$ to $b_2'$ included in $o'$ that does not contain any edges of torso paths. 
As $Q'$ has both its endvertices on the same side of the cut $W$, it intersects that cut 
evenly. 
The edges of $Q'$ form 
a path $Q$ in the graph $G$ from $b_1$ to $b_2$.
By definition of local splitting, there is a cycle $o$ of $G$ of the same length as $o'$ that 
includes the path $Q$. So by \autoref{unique_copy_extended}, $Q$ is a path in $\expl(b_1,b_2)$.  
\end{proof}

Suppose for a contradiction that the edges $e$ and $f$ are not separated by $\{b_1,b_2\}$ 
in $G$; that is, they are incident with vertices of the same component 
of the 
punctured explorer-neighbourhood $\expl(b_1,b_2)-b_1-b_2$.

\begin{sublem}\label{o_exists}
 There is a cycle $o$ of the explorer-neighbourhood $\expl(b_1,b_2)$ in $G$ that intersects the set 
$W$ oddly.
\end{sublem}
\begin{proof}
By assumption, there is a path $P$ included in the punctured 
explorer-neighbourhood $\expl(b_1,b_2)-b_1-b_2$ between the endvertices of the edges $e$ and 
$f$ outside $\{b_1,b_2\}$. 
Now we extend the path $P$ to a walk by adding the edges $e$ and 
$f$ at the endvertices of $P$. The endvertices of this extended walk are in the set $\{b_1,b_2\}$. 
This walk intersects the edge set $W$ precisely in the edge $e$. 
Either this walk is a cycle, or it is a path whose endvertices are $b_1$ and $b_2$. 
While we are done immediately in the first case, in the second case 
we concatenate this path $ePf$ with a path $Q$ as in  
\autoref{path_exists1}. This way we obtain a closed walk, which includes the desired cycle $o$. 
\end{proof}

\begin{sublem}\label{o1_exists}
There is a cycle $o_1$ of $G$ contained in the explorer-neighbourhood $\expl(b_1,b_2)$ of 
length bounded by $r$ that intersects the set 
$W$ oddly.
\end{sublem}
\begin{proof}
 Let $o$ be a cycle as in \autoref{o_exists}. By \autoref{cycle_gen}, the cycle $o$ is generated 
from cycles that are included within the balls of radius $r/2$ around the vertices $b_1$ and 
$b_2$. These cycles, in turn by \autoref{gen}, are generated by cycles within the respective balls 
of length bounded by $r$. To summarise: the cycle $o$ is generated over the finite field 
$\Fbb_2$ by cycles of $G$ contained in the explorer-neighbourhood $\expl(b_1,b_2)$ of 
length bounded by $r$. As the cycle $o$ intersects the set $W$ oddly, one of the cycles in 
the generating set has to intersect the set $W$ oddly. We pick such a cycle for $o_1$.  
\end{proof}

\begin{sublem}\label{oprime_exists}
 There is a cycle $o'$ of $G'$ included in the explorer-neighbourhood $\expl(b_1',b_2')$ of 
length bounded by 
$r$ that intersects the set 
$W$ oddly.
\end{sublem} 

\begin{proof}
Let $o_1$ be a cycle as in \autoref{o1_exists}. First assume the cycle $o_1$ does not traverse the 
local 2-separator $\{a_1,a_2\}$ (here we say that a cycle \emph{traverses} a local 2-separator if 
the cycle traverses\footnote{Given a cycle $o$ and separator $S$ in a graph $G$, a \emph{traversal} 
of $o$ is a minimal subpath of $o$ such that the vertex just after and just before that subpath are 
in different components of $G\sm S$. A cycle traverses a 2-separator in exactly two subpaths, each 
being a single vertex, or not at all.} this set as a separator of the 
explorer-neighbourhood [of parameter 
$r$]). Then the cycle $o_1$ of $G$ is a cycle of $G'$.
So we can take $o'=o_1$ and are done.
Hence we may assume, and we do assume, that the cycle $o_1$ traverses the local 2-separator 
$\{a_1,a_2\}$. We remark that as the cycle $o_1$ is a cycle of the graph $G$ -- not just of the 
explorer-neighbourhood -- it cannot contain two copies of a vertex of the graph $G$. 

One of the subpaths of $o_1$ from $a_1$ to $a_2$ contains an even number of 
edges of $W$, the other one an odd number of edges of $W$.
Let $P$ be the subpath of $o_1$ from $a_1$ to $a_2$ that contains an odd number of edges of $W$.
Then the edges of $P$ form a path of the graph $G'$ from a slice of $a_1$ to a slice of 
$a_2$. 
We obtain the cycle $o'$ from $P$ by adding a torso path (which is added by the definition of 
local cutting between the two slices $a_1'$ and $a_2'$), which is does not intersect $W$ by 
the choice of 
the 
component $k'$. \end{proof}

The edge set $W$ is a cut of the explorer-neighbourhood $\expl(b_1',b_2')$, and the cycle $o'$, 
which is given by \autoref{oprime_exists}, is a 
cycle of that graph. So they must intersect evenly (as all cuts and cycles do). This is a direct 
contradiction to \autoref{oprime_exists}. Hence the punctured explorer-neighbourhood 
$\expl(b_1,b_2)-b_1-b_2$ in 
$G$ is disconnected, and so $\{b_1,b_2\}$ is an $r$-local separator of the graph $G$. More 
specifically, the edges $e$ and $f$ are separated by $\{b_1,b_2\}$. 
\end{proof}

\begin{rem}\label{rem77} Here we come back to \autoref{expl_explained}. 
If the punctured double ball around two vertices is disconnected, then so is the punctured 
explorer-neighbourhood around them. We say that an $r$-local separator $\{v,w\}$ is \emph{mundane} 
if the punctured double-ball around $v$ and $w$ is disconnected.   \myref{recur_gd0_ONESTEP} says 
that the class of local 2-separators has the corner property, a property which plays a central 
role in the structural theory of genuine separators.
The subclass of mudane local 2-separators does not have this property, which can be seen from 
\autoref{fig:no-corner} as follows.
If one obtained $G'$ 
from the depicted graph by locally cutting at the mudane local separator $\{b_1,b_2\}$, the corner 
$\{a_1,b_1\}$ becomes a mudane local separator, which does not come from a mudane local separator 
of 
$G$.
\end{rem}

   \begin{figure} [htpb]   
\begin{center}
   	  \includegraphics[height=6cm]{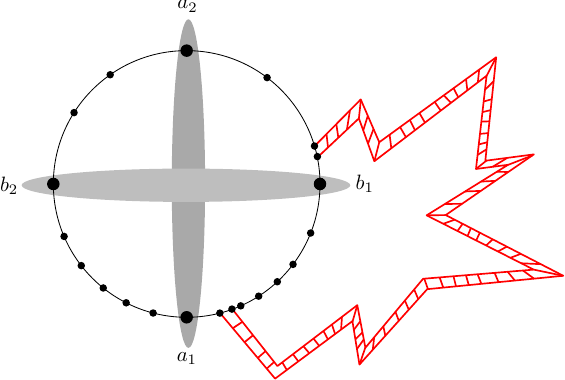}
   	  \caption{Two crossing local separators. They are highlighted in grey and denoted by 
$\{a_1,a_2\}$ and $\{b_1,b_2\}$. The long red strip joins a neighbour of $a_1$ with a neighbour of 
$b_1$. Its length is long enough so that the punctured 
double-balls $(B_{r/2}(a_1)\cup B_{r/2}(a_2))-a_1-a_2$ and $(B_{r/2}(b_1)\cup 
B_{r/2}(b_2))-b_1-b_2$ are 
disconnected but so short that the punctured 
double-ball $(B_{r/2}(a_1)\cup B_{r/2}(b_1))-a_1-b_1$ is connected.} \label{fig:no-corner}
\end{center}
   \end{figure}

\begin{eg}
While \autoref{fig:no-corner}, already shows that a notion of local 2-separators based on punctured 
double balls would not have all the desired properties, here we give a second more advanced example 
that demonstrates that even more properties fail for the class of mudane local separators. 

We start by giving a formal definition of the graph $G$ depicted in \autoref{fig:moebius-corner}.
Given a parameter $r\geq 6$, let $M$ be the graph obtained from a cycle $o$ of length $2r$ by 
adding a paths of length two between two antipodal vertices of $o$ (that is, any two vertices of 
$o$ of distance precisely $r$). We obtain $G$ from $M$ by picking an edge $e$ of $o$ arbitrarily 
and subdividing it four times. 
Informally speaking, the graph $G$ is a subdivided Moebius strip, see \autoref{fig:moebius-corner}.

Now we give names to four vertices of $G$. Denote the two endvertices of the edge $e$ by $a_2$ and 
$b_2$. Denote the antipodal vertex of the vertex $a_2$ on $o$ by $a_1$, and let $b_1$ be the 
subdivision vertex of the edge $e$ that has distance two from $a_2$ (and consequently distance 
three from $b_2$). 

We observe (and prove below) that:
\begin{enumerate}[(1)]
 \item $B_{\frac{r}{2}}(a_1)\cup B_{\frac{r}{2}}(b_1)-a_1-b_1$ is disconnected;
 \item $B_{\frac{r}{2}}(b_1)\cup B_{\frac{r}{2}}(b_2)-b_1-b_2$ is disconnected ALTHOUGH;
  \item $B_{\frac{r}{2}}(a_1)\cup B_{\frac{r}{2}}(a_2)-a_1-a_2$ is connected AND;
   \item $B_{\frac{r}{2}}(a_1)\cup B_{\frac{r}{2}}(b_2)-a_1-b_2$ is connected;
\end{enumerate}

To see (1) note that the ball $B_{\frac{r}{2}}(b_1)$ is included in the ball $B_{\frac{r}{2}}(a_1)$ 
and $B_{\frac{r}{2}}(a_1)-a_1-b_1$ is disconnected. 
Condition (2) follows from the fact that $\{b_1,b_2\}$ is a global 2-separator. 
To see (4), just note that the union of these two balls includes the whole graph $G$; and 
similarly to see (3) note that the union of the two relevant balls is equal to 
the graph $G$ with a single vertex removed. 

Similarly as  \autoref{fig:no-corner}, this example shows that the corner property does not in 
general hold for the class of mudane local separators. Also the following rather obscure thing 
happens. 
Let $G'$ be the graph obtained from $G$ by $r$-locally cutting at the local 2-separator 
$\{a_1,b_1\}$. Then $B_{\frac{r}{2}}(a_1)\cup B_{\frac{r}{2}}(a_2)-a_1-a_2$ is disconnected in the 
graph $G'$, while (3) above says that this property does not hold in $G$. It seems like 
$\{a_1,a_2\}$ should be an $r$-local 2-separator in any sensible $r$-local decomposition for 
$G$.
\end{eg}

   \begin{figure} [htpb]   
\begin{center}
   	  \includegraphics[height=3cm]{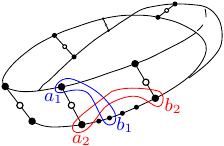}
   	  \caption{A subdivided Moebius strip. } \label{fig:moebius-corner}
\end{center}
   \end{figure}

\section{When all local 2-separators are crossed...}\label{sec:crossed}

In this section we prove \autoref{structure_torso_intro}, which later will be used in the proof of 
\autoref{main_2sepr-intro}.

\begin{proof}[Proof of \autoref{structure_torso}.]
 Let $r\in \Nbb\cup\{\infty\}$ and let $G$ be a connected graph that is $r$-locally 2-connected.
Assume that every $r$-local 2-separator of $G$ is crossed by an $r$-local 2-separator. 
We are to show that $G$ is $r$-locally $3$-connected or a cycle of length at most $r$.

Let $\{a_1,a_2\}$ be an $r$-local 2-separator, and let $\{b_1,b_2\}$ be an $r$-local 
2-separator that crosses it. By \myref{alt_exist}, there is a cycle $o$ of length at most 
$r$ alternating between these two local separators. 
Our aim is to show that the graph $G$ is equal to the cycle $o$. 
Suppose not for a contradiction. 
Then there 
is a vertex outside the cycle $o$. As the graph $G$ is connected, there is a vertex that is outside 
$o$ and adjacent to a vertex on the cycle $o$. Pick such a vertex and call it $x_2$, and denote 
one of its neighbours on the cycle $o$ by $x_1$. 
We distinguish two cases.

{\bf Case 1:} there is no vertex $y$ on the cycle $o$ such that $\{x_1,y\}$ is an $r$-local 
2-separator.

By $\Scal$ we denote the set of $r$-local 
2-separators with both their vertices on the cycle $o$. 
Given $\{c_1,c_2\}\in \Scal$, by \autoref{unique_copy_extended} all 
vertices on the cycle $o$ have a unique copy in the explorer-neighbourhood $\expl(c_1,c_2)$.
For these vertices we suppress a bijection between them and their unique copies in $\expl(c_1,c_2)$ 
to simplify notation.  
By $\Gamma(c_1,c_2)$ we denote the set of vertices on the cycle $o$ in the component of $o-c_1-c_2$ 
that contains the vertex $x_1$. 
The \emph{size} of a local separator $\{c_1,c_2\}$ in $\Scal$ is $|\Gamma(c_1,c_2)|$.
The set $\Scal$ is nonempty as $\{a_1,a_2\}\in \Scal$. 
Pick a local separator $\{w_1,w_2\}\in \Scal$ of minimal size.

\begin{sublem}\label{not_crossed1}
In Case 1, no $r$-local 2-separator crosses $\{w_1,w_2\}$. 
\end{sublem}

\begin{proof}
Suppose for a contradiction there is an $r$-local 2-separator $\{v_1,v_2\}$ that crosses 
$\{w_1,w_2\}$.
By \myref{alt_exist} there is a cycle $o'$ alternating between $\{v_1,v_2\}$ and $\{w_1,w_2\}$. 
Hence by \autoref{unique_copy_extended} the vertices $v_1$, $v_2$, $w_1$ and $w_2$ have unique 
copies in $\expl(v_1,v_2)$; and so in the following we will suppress a bijection between them and 
their copies in $\expl(v_1,v_2)$ from our notation. 
Let $P'$ be a subpath of $o'$ between $w_1$ and $w_2$ of length at most $r/2$.
Let $P$ be a subpath of $o$ between $w_1$ and $w_2$ of length at most $r/2$.
Let $o''$ be the closed walk obtained by concatenating $P$ and $P'$. The path $P'$ must contain one 
of the 
vertices $v_1$ or $v_2$, say $v_1$. Hence $o''$ is a closed walk through $v_1$ of length at most 
$r$; so 
it is contained within $B_{r/2}(v_1)$. So the path $P$ is a path of the ball $B_{r/2}(v_1)$. As 
$w_1$ and 
$w_2$ are in different components of the punctured explorer-neighbourhood $\expl(v_1,v_2)-v_1-v_2$, 
the path $P$ must contain the vertex $v_1$ or $v_2$. 
Thus the cycle $o$ contains the vertex $v_1$ or $v_2$. By condition (2) of 
\myref{alt_exist}, the cycle $o$ alternates between the local separators  $\{v_1,v_2\}$ and
$\{w_1,w_2\}$. 

By \autoref{cross_sym} each of the punctured explorer-neighbourhoods 
$\expl(v_1,v_2)-v_1-v_2$ and $\expl(w_1,w_2)-w_1-w_2$ has precisely two components. Hence there is 
a corner $\{v_i,w_j\}$ with $i,j\in \{1,2\}$ such that the vertex $x_1$ is a person of type one 
living in 
the corner $\{v_i,w_j\}$.
Hence by \myref{corner-lemma}, $\{v_i,w_j\}$ is an $r$-local 2-separator. It is in the set 
$\Scal$.
We claim that its size is strictly smaller than that of $\{w_1,w_2\}$. Indeed, by the 
`Moreover'-part of \myref{corner-lemma}, $\Gamma(v_i,w_j)$ only contains those vertices of $o$ on 
the path between $v_i$ and $w_j$ containing $x_1$. As the vertices $x_1$ and $v_i$ are in the 
same component of the punctured explorer-neighbourhood $\expl(w_1,w_2)-w_1-w_2$, all vertices 
of $\Gamma(v_i,w_j)$ are also in $\Gamma(w_1,w_2)$ but that set additionally contains the  vertex 
$v_i$. This is the desired contradiction.
Thus no $r$-local 2-separator crosses $\{w_1,w_2\}$. 
\end{proof}

\autoref{not_crossed1} contradicts the assumptions of the theorem. Hence the graph $G$ is a cycle. 
Having finished Case 1, it remains to treat the following (which will be somewhat similar).

{\bf Case 2: } not Case 1; that is, there is a vertex $y$ on the cycle $o$ such that $\{x_1,y\}$ 
is an $r$-local 2-separator.
 
By $\Scal$ we denote the set of $r$-local 
2-separators with both their vertices on the cycle $o$ and one of these vertices is equal to the 
vertex $x_1$. 
Given $\{c_1,c_2\}\in \Scal$, by \autoref{unique_copy_extended} all 
vertices on the cycle $o$ have a unique copy in the explorer-neighbourhood $\expl(c_1,c_2)$. 
For these vertices we suppress a bijection between them and their unique copies in $\expl(c_1,c_2)$ 
to 
simplify notation. 
Moreover, the edge $x_1x_2$ of $G$ has a unique copy in $\expl(c_1,c_2)$ by \autoref{unique_copy}. 
We suppress a bijection between the vertex $x_2$ of $G$ and the endvertex of the edge $x_1x_2$ in 
$\expl(c_1,c_2)$ that is a copy of $x_2$.

By $\Gamma(c_1,c_2)$ we denote the set of vertices on the cycle $o$ in the component of the 
punctured explorer-neighbourhood $\expl(c_1,c_2)-c_1-c_2$ that contains the vertex $x_2$. 
The \emph{size} of a local separator $\{c_1,c_2\}$ in $\Scal$ is $|\Gamma(c_1,c_2)|$.
The set $\Scal$ is nonempty by the assumption of Case 2. 
Pick a local separator $\{w_1,w_2\}\in \Scal$ of minimal size.

Arguing the same as in the proof of \autoref{not_crossed1} but referring to the fact that `$x_2$ is 
a 
person of type two' instead of `$x_1$ is a person of type one', one proves the 
following\footnote{By local 2-connectivity, if there is no person of type one but a person of type 
two, there is also a person of type three. So we need not to rely on persons of type three in our 
arguments.}.

\begin{sublem}\label{not_crossed2}
In Case 2, no $r$-local 2-separator crosses $\{w_1,w_2\}$. \qed
\end{sublem}
 
This completes all the cases. Hence in all cases, there is an $r$-local 2-separator that is not 
crossed by any other $r$-local 2-separator. This is a contradiction to the assumptions of this 
theorem. 
  Hence the graph $G$ must be equal to the cycle $o$.  
\end{proof}

\section{The uniqueness statement of the local 2-separator theorem}\label{sec:uni}

In this section we prove \autoref{main_2sepr-intro}. 
Our first goal is to prove 
\autoref{lifting-lem} below, which can be seen as the `inverse' 
of \myref{recur_gd0_ONESTEP}.

Let $G'$ be a graph obtained from a graph $G$ by $r$-locally cutting an $r$-local $2$-separator 
$\{a_1,a_2\}$. Let $\{b_1,b_2\}$ be an $r$-local $2$-separator of $G$. 

\begin{lem}\label{lift}
Assume $G$ is $r$-locally 2-connected. 
Assume $\{b_1,b_2\}$ is not crossed by $\{a_1,a_2\}$. 
Then there is a cycle $o'$ of the graph $G'$ of length at most $r$ that contains vertices 
$b_i'$ that are equal to $b_i$ or a slice of $b_i$ (for $i=1,2$). 

Moreover, for any $i\in \{1,2\}$ with $b_i\not\in \{a_1,a_2\}$ 
the two edges of $o'$ incident with $b_i'$ are separated by $\{b_1',b_2'\}$ -- or else the edge 
$b_1'b_2'$ is on $o'$.
\end{lem}

We remark that the neighbours of $b_i'$ are also neighbours of $b_i$ (unless the edge joining them 
is on a torso path), and they have unique copies in $\expl(b_1,b_2)$ by 
\autoref{unique_copy}; 
in this sense the `Moreover'-part is unambiguously defined. 

\begin{proof}[Proof of \autoref{lift}.]
As the graph $G$ is $r$-locally 2-connected, we can apply \myref{local_is_very_local} to deduce 
that there is a cycle $o$ of length at 
most $r$ in $G$ through the vertices $b_1$ and $b_2$ and such that interior vertices of 
different subpaths of $o$ between $b_1$ and $b_2$ are in different components of the punctured 
explorer-neighbourhood $\expl(b_1,b_2)-b_1-b_2$ -- or else the cycle $o$ contain the edge 
$b_1b_2$. 

If the cycle $o$ does not contain any vertex $a_i$, it is a cycle of the graph $G'$ and we are 
done. So we may assume, and we do assume, that a vertex $a_i$, say $a_1$, is on the cycle $o$. 
So the cycle $o$ is a cycle of the explorer-neighbourhood $\expl(a_1,a_2)$. 
In there,  $\{a_1,a_2\}$ is a genuine 2-separator. There are three cases: if $o$ does not contain 
any of its vertices, then $o$ is a 
cycle of the graph $G'$ and we are done. If $o$ contains precisely one vertex, all other vertices 
must be in a single local component and so $o$ is a 
cycle of the graph $G'$ and we are done. Hence it remains to consider the third case that both 
vertices $a_1$ and $a_2$ are on the cycle $o$.  As the local 
separator $\{a_1,a_2\}$ does not cross  $\{b_1,b_2\}$, the vertices $a_1$ and $a_2$ must be in the 
same component of the punctured explorer-neighbourhood $\expl(b_1,b_2)-b_1-b_2$. In particular, the 
cycle $o$ does not pre-alternate between the local separators $\{a_1,a_2\}$ and 
$\{b_1,b_2\}$. So there is a subpath $P$ of $o$ between the $a_i$ containing no vertex $b_j$ as an 
interior 
vertex. We obtain $o'$ from $o$ by replacing the path $P$ by a torso path to obtain a cycle of 
the graph $G'$. By the definition of the length of the torso paths the 
length of $o'$ is at most 
that of $o$. And $o'$ contains the vertices $b_i$ or slices thereof. 
This completes the proof except for the `Moreover'-part.

To see the `Moreover'-part, pick $b_i$ under the constraint that it is not in $\{a_1,a_2\}$. The 
two incident edges of the vertex 
$b_i'$ on $o'$ are not edges on torso paths. By the construction of the cycle $o'$, 
the two neighbours of $b_i'$ on $o'$ are separated by $\{b_1',b_2'\}$.
\end{proof}

\begin{dfn}[Lift]
 A \emph{lift} of a local 2-separator $\{b_1,b_2\}$ of $G$ is a set $\{b_1',b_2'\}$ of $G'$ such 
that each vertex 
$b_i'$ is equal to $b_i$ or a slice thereof, and such that there is a cycle $o'$ of length at 
most $r$ containing $b_1'$ and $b_2'$.  \addcontentsline{toc}{subsection}{Dfn: Lift}
\end{dfn}

\begin{rem}\label{lift_unique}
Under the circumstances of \autoref{lift}, it can be shown that for each local 2-separator 
$\{b_1,b_2\}$ a lift is \emph{uniquely} defined -- unless $\{b_1,b_2\}$ is identical to 
$\{a_1,a_2\}$ (note that $\{a_1,a_2\}$ does not cross itself). 
Indeed, if no $b_i$ is in the set $\{a_1,a_2\}$, this is clear. Otherwise there can be at most one 
vertex 
$b_i$ that is in  $\{a_1,a_2\}$, say it is $b_1$. A vertex $b_1'$ in a lift has distance at most 
$r/2$ from $b_2=b_2'$. As any other slice $x'$ of $b_1$ has distance at least $r+1$ from $b_1'$ 
by \autoref{cut_far}, the vertex $b_2$ has a too large distance from that vertex. So $\{x',b_2\}$ 
cannot be a lift. 
Thus we will in the following always refer to `the' lift. 
\end{rem}

\begin{lem}[Lifting Lemma]\label{lifting-lem}
Assume $G$ is $r$-locally 2-connected. 
Let $\{b_1,b_2\}$ be an $r$-local $2$-separator 
 of $G$. Assume $\{b_1,b_2\}$ is not crossed by $\{a_1,a_2\}$ and they are not identical.
Then the lift $\{b_1',b_2'\}$ of $\{b_1,b_2\}$ is an $r$-local $2$-separator of $G'$. 

More specifically, if edges $e$ and $f$ with precisely one endvertex in $\{b_1,b_2\}$ are 
separated by $\{b_1,b_2\}$ in $G$, then 
their contacts are separated by $\{b_1',b_2'\}$ in $G'$.
\end{lem}

\begin{proof}
 The proof strategy is somewhat similar to that of 
\myref{recur_gd0_ONESTEP}.
 Like in the proof of \autoref{recur_gd0_ONESTEP}, in this proof we will distinguish 
between the vertices of $G$ and $G'$ by adding a dash to the 
vertices of the graph $G'$. 

Let $e$ and $f$ be edges incident with 
precisely one of $b_1$ or $b_2$ that are separated by $\{b_1,b_2\}$ in $G$. 

\begin{sublem}\label{torso-lemma2}
 There is a single component $k$ of $\expl(b_1,b_2)-b_1-b_2$ containing an endvertex of every edge  
incident with precisely one of $b_1$ or $b_2$ whose contact is an edge of a torso path.  
\end{sublem}

\begin{proof}
 By assumption not both vertices $b_1$ and $b_2$ are in the set $\{a_1,a_2\}$. As we are done 
otherwise, we assume that precisely one of the vertices $b_i$ is in $\{a_1,a_2\}$.
By symmetry, we assume that $b_1=a_1$. 

Next, we determine the component of $\expl(a_1,a_2)-a_1-a_2$ belonging to the slice $b_1'$ of 
$a_1$. 
As $\{b_1,b_2\}$ is an $r$-local 2-separator, by 
\myref{local_is_very_local}, there is a cycle $o$ of length at most $r$ through $b_1$ and $b_2$.
By 
\autoref{unique_copy_extended}, the vertex $b_2$ has a unique copy in $\expl(a_1,a_2)$. 
So there is a unique slice of the vertex $a_1$ that has distance at most $r/2$ from the vertex 
$b_2=b_2'$. This is the slice for the component of $\expl(a_1,a_2)-a_1-a_2$ containing $b_2$. 
Denote that component by $k_2$. 
Hence the vertex $b_1'$ is the slice of the vertex $a_1$ for the component $k_2$. 

Next we define the component $k$.
 As $\{a_1,a_2\}$ is an $r$-local 2-separator, by 
\myref{local_is_very_local}, there is a cycle of length at most $r$ through $a_1$ and $a_2$. By 
\autoref{unique_copy_extended}, the vertex $a_2$ has a unique copy in $\expl(b_1,b_2)$. 
Let $k$ be the component of $\expl(b_1,b_2)-b_1-b_2$ that contains the vertex $a_2$.

Now let $e$ be an edge of $G$ incident with precisely one of $b_1$ or $b_2$ whose contact is an 
edge of a torso path.  Then $e$ is incident with the vertex $b_1=a_1$. 
Let $x$ be the endvertex of the edge $e$ aside from $b_1$.
As the contact for $e$ is an edge on a torso path, the vertex $x$ is outside the component 
$k_2$ of 
$\expl(a_1,a_2)-a_1-a_2$. As $G$ is $r$-locally $2$-connected, the punctured ball 
$B_{r/2}(a_1)-a_1$ is connected. So there is a path $P$ from $x$ to $a_2$ within that ball. As the 
vertex $b_2$ is in a different connected component of $\expl(a_1,a_2)-a_1-a_2$ than $x$, it is also 
in a different connected component of $B_{r/2}(a_1)-a_1-a_2$ than $x$. So the path $P$ does not 
contain the vertex $b_2$.
Thus $P$ is a path from $x$ to $a_2$ included in the component $k$.
Note that $P$ is a path of $\expl(b_1,b_2)$; indeed, it is within the ball of radius $r/2$ around 
$b_1$ (which is equal to $a_1$ here). 
 So the path $P$ witnesses that the vertex $x$ is in the component $k$ of 
$\expl(b_1,b_2)-b_1-b_2$. 

As the edge $e$ was arbitrary, the component $k$ contains an endvertex of every edge  
incident with precisely one of $b_1$ or $b_2$ whose contact is a torso edge.  
\end{proof}

Let $k_1$ be the component of the punctured explorer-neighbourhood $\expl(b_1,b_2)-b_1-b_2$ 
containing an endvertex of the edge $e$. By exchanging the roles of $e$ and $f$ if necessary, we 
assume that the component $k_1$ is different from the component $k$ in \autoref{torso-lemma2}. 
Let  $W$ be the set of edges of $G$ with one endvertex in $\{b_1,b_2\}$ and the other endvertex 
in $k_1$. By  \autoref{torso-lemma2}, the edge set $W$ does not contain an edge of a torso path 
and hence is 
an edge set of the graph $G'$ consisting of edges with precisely one endvertex in the set 
$\{b_1',b_2'\}$.

\begin{sublem}\label{path_exists2}
 There is a path $Q'$ from $b_1'$ to $b_2'$ contained in $\expl(b_1',b_2')$ that contains an 
even number of edges from $W$. 
\end{sublem}

\begin{proof}
As $G$ is $r$-locally 2-connected, by \myref{local_is_very_local} there is a cycle $o$ included 
in 
$\expl(b_1,b_2)$ containing the vertices $b_1$ and $b_2$ of length at most $r$, and $o$ 
contains vertices in two different components of $\expl(b_1,b_2)-b_1-b_2$ or an edge from $b_1$ to 
$b_2$. 
In the second case, we immediately set $Q$ to be the edge from $b_1$ to $b_2$. In the first case we 
define $Q$ as follows. 
As $\{a_1,a_2\}$ does not cross  $\{b_1,b_2\}$ by condition (4) in \myref{alt_exist}, one of 
the two subpaths of $o$ from $b_1$ to $b_2$ has no vertex $a_i$ as interior vertex.
Pick such a 
path and call it $Q$.  As $\{a_1,a_2\}\neq \{b_1,b_2\}$, one of the vertices $b_i'$ is equal to 
$b_i$. Hence the cycle $o$ is in the ball of radius $r/2$ around that vertex, and so in the 
explorer-neighbourhood $\expl(b_1',b_2')$. Thus the subpath $Q$ is within $\expl(b_1',b_2')$.

As $Q$ has both its endvertices on the same side of the cut $W$, it intersects that cut 
evenly. 
The edges of $Q$ form 
a path $Q'$ in the graph $G'$ from $b_1'$ to $b_2'$, which is also a path in $\expl(b_1',b_2')$ as 
shown above.  
\end{proof}

Suppose for a contradiction that the contacts for the edges $e$ and $f$ are not separated by 
$\{b_1',b_2'\}$ 
in $G'$; that is, they are incident with vertices of the same component 
of the 
punctured explorer-neighbourhood $\expl(b_1',b_2')-b_1'-b_2'$.

\begin{sublem}\label{o_exists2}
 There is a cycle $o'$ of the explorer-neighbourhood $\expl(b_1',b_2')$ in $G'$ that intersects the 
set 
$W$ oddly.
\end{sublem}
\begin{proof}
By assumption, there is a path $P'$ included in the punctured explorer-neighbourhood between the 
endvertices of the contacts for the edges $e$ and $f$ outside $\{b_1',b_2'\}$.
Now we extend the path $P'$ to a walk by adding the contacts for the edges $e$ and $f$ at the two 
ends. This walk intersects the set $W$ precisely in the edge\footnote{\autoref{torso-lemma2} 
implies that one of the edges $e$ and $f$ is its own contact. By fixing the roles of 
$e$ and $f$ above in the definition of the set $W$, we ensured that the edge $e$ is its own 
contact.} $e$.
If we added the same vertex to both ends, we obtained the desired cycle $o'$.

Otherwise we obtain a path between the vertices $b_1'$ and $b_2'$ in the explorer-neighbourhood 
that 
intersects the set $W$ oddly.
Concatenating this path with a path $Q'$ as in \autoref{path_exists2} yields a closed walk 
intersecting oddly. This closed walk is an edge-disjoint union of cycles and one of those cycles 
intersects $W$ oddly.  
\end{proof}

\begin{sublem}\label{o1_exists2}
There is a cycle $o_1'$ of $G'$ contained in the explorer-neighbourhood $\expl(b_1',b_2')$ of 
length bounded by $r$ that intersects the set 
$W$ oddly.
\end{sublem}
\begin{proof}
 Let $o'$ be a cycle as in \autoref{o_exists2}. By \autoref{cycle_gen}, the cycle $o'$ is 
generated 
from cycles of that are included within the balls of radius $r/2$ around the vertices $b_1'$ 
and 
$b_2'$. These cycles, in turn by \autoref{gen}, are generated by cycles within the respective balls 
of length bounded by $r$. To summarise: the cycle $o'$ is generated over the finite field 
$\Fbb_2$ by cycles of $G'$ contained in the explorer-neighbourhood $\expl(b_1',b_2')$ of 
length bounded by $r$. As the cycle $o'$ intersects the set $W$ oddly, one of the cycles 
in 
the generating set has to intersect the set $W$ oddly. We pick such a cycle for $o_1'$.  
\end{proof}

\begin{sublem}\label{oprime_exists2}
 There is a cycle $o$ of $G$ of 
length bounded by 
$r$ that intersects the set 
$W$ oddly. Moreover, $o$ is a cycle of $\expl(b_1,b_2)$. 
\end{sublem} 

\begin{proof}
Let $o_1'$ be a cycle as in \autoref{o1_exists2}. We remark that as the cycle $o_1'$ is a cycle of 
the graph $G'$ -- not just of the 
explorer-neighbourhood -- it cannot contain two copies of a vertex of the graph $G'$. 
As the cycle $o_1'$ has length at most $r$, by \autoref{cut_far} it contains at most one slice of 
any vertex of $G$. In particular, the cycle $o_1'$ contains at most one torso path. 

First assume the cycle $o_1'$ 
does not use any edges on torso paths. Then the edges of the cycle $o_1'$ of $G'$ form a cycle 
$o$ of $G$.
So we are done in this case.

Hence we may assume, and we do assume, that the cycle $o_1'$ contains a unique torso path. 
We denote by $Q'$ the subpath of $o_1'$ obtained by removing this torso path. Then $Q'$ bijects 
to 
an edge set $Q$ of the graph $G$. As the torso path is not in $W$ by construction, we 
deduce that $Q$ 
contains an odd number of edges of $W$.

By the definition of local cutting, there is a path $P$ of\footnote{Formally, the 
path $P$ lives in the 
explorer-neighbourhood $\expl(a_1,a_2)$ and in there it is included in a cycle of length at most 
$r$ and by \autoref{new_rem}, this cycle is a cycle of $G$, and so $P$ is a path of $G$.} 
$G$ between $a_1$ and $a_2$ associated to the torso path of $o_1'$
such that the cycle $o_1'$ with the torso path replaced by $P$ is a cycle of the graph $G$ of 
the 
same length as $o_1'$.
We denote that cycle by $o$.

To see the \lq Moreover\rq-Part note that the cycle $o$ contains at least one edge of $W$ and so 
must contain a vertex $b_i$ and so is in the ball of radius $r/2$ around that vertex and thus is 
 embedded in the explorer-neighbourhood $\expl(b_1,b_2)$. From now on we consider this embedding of 
$o$ in $\expl(b_1,b_2)$ (note that this is unambiguous as if $o$ contained both vertices $b_i$, 
then the embedding would be unique by \autoref{unique_copy_extended}). 

As $Q$ contains an odd number of edges of $W$, it remains to show that the path $P$ contains an 
even number of edges from the set $W$.
Denote by $a_1'$ and $a_2'$ the copies of the vertices $a_1$ and $a_2$ 
on $o$, respectively. 
The existence of $o$ implies that $\{a_1,a_2\}$ does 
not pre-cross $\{b_1,b_2\}$; that is, the copies $a_1'$ and $a_2'$ are in the same component of 
$\expl(b_1,b_2)-b_1-b_2$. So they are on the same side of the cut $W$.
So by \autoref{traverse_standard2}, $P$ contains an even number of edges of $W$. 
Thus the cycle $o$, which is composed of the paths $P$ and $Q$ contains an 
odd number of edges of $W$.  
The \lq Moreover\rq-Part has been shown in the text, so this completes the proof. 
\end{proof}

The edge set $W$ is a cut of the explorer-neighbourhood $\expl(b_1,b_2)$, and a cycle $o$ as 
in \autoref{oprime_exists2} is a 
cycle of that graph. So they must intersect evenly (as all cuts and cycles do). This is a direct 
contradiction to \autoref{oprime_exists2}. Hence the punctured explorer-neighbourhood 
$\expl(b_1',b_2')-b_1'-b_2'$ in 
$G'$ is disconnected, and so $\{b_1',b_2'\}$ is an $r$-local separator of the graph $G'$. 
More specifically, the contacts for $e$ and $f$ are separated by $\{b_1',b_2'\}$.
\end{proof}

\vspace{.3cm}

\begin{lem}\label{lift_non-crossing}Let $G$ be an $r$-locally 2-connected graph.
Let $\{a_1,a_2\}$, $\{b_1,b_2\}$ and $\{c_1,c_2\}$ be $r$-local 2-separators 
so that $\{a_1,a_2\}$ crosses neither $\{b_1,b_2\}$ nor $\{c_1,c_2\}$.
Construct $G'$ from $G$ by $r$-locally cutting $\{a_1,a_2\}$.
 
 Then the lifts of $\{b_1,b_2\}$ and $\{c_1,c_2\}$ cross in $G'$ if and only if  
 $\{b_1,b_2\}$ and $\{c_1,c_2\}$ cross in $G$.
 \end{lem}

\begin{proof}
Assume  $\{b_1,b_2\}$ and $\{c_1,c_2\}$ cross in $G$. Then by condition (2) of  
\myref{alt_exist} there is a cycle $o$ of length at most $r$ alternating between $\{b_1,b_2\}$ 
and $\{c_1,c_2\}$ traversing the local separator $\{b_1,b_2\}$ twice (that is, the subpaths of $o$ 
between $b_1$ and $b_2$ have interior vertices in 
different components of the 
punctured explorer-neighbourhood $\expl(b_1,b_2)-b_1-b_2$). 

If the cycle $o$ does not traverse the local separator $\{a_1,a_2\}$ twice, then $o$ is a cycle of 
the graph $G'$ (by the argument already given in the proof of \autoref{lift}). Then by 
\myref{lifting-lem}, $o$ traverses the lift of $\{b_1,b_2\}$ twice. Thus 
the lifts of $\{b_1,b_2\}$ and $\{c_1,c_2\}$ cross in $G'$.

Hence we may assume, and we do assume, that the cycle $o$ traverses the local separator 
$\{a_1,a_2\}$ twice. As $\{a_1,a_2\}$ crosses neither $\{b_1,b_2\}$ nor $\{c_1,c_2\}$, 
one $a_1$-$a_2$-subpath of $o$ must contain both vertices $b_1$ and $b_2$. As the cycle $o$
alternates, this path must also contain one vertex $c_i$ between $b_1$ and $b_2$ and it must 
contain the second vertex $c_i$ as $\{a_1,a_2\}$ does not cross $\{c_1,c_2\}$ (in the form of 
condition (4) of \myref{alt_exist}). So there is an $a_1$-$a_2$-subpath $P$ of $o$ containing none 
of the vertices $b_i$ or $c_j$ as interior vertices. 
Let $o'$ be the cycle obtained from $o$ by replacing the path $P$ by a torso path. 
The cycle $o'$ contains a lift of $\{b_1,b_2\}$ by construction. As the lift is unique by 
\autoref{lift_unique}, it must contain the lift of $\{b_1,b_2\}$. 
The cycle $o'$ traverses the lift of $\{b_1,b_2\}$ twice by the `More 
specifically'-part of \myref{lifting-lem}.
Similarly, $o'$ traverses the lift of $\{c_1,c_2\}$ twice. So it alternates 
between the lifts of $\{b_1,b_2\}$ 
and $\{c_1,c_2\}$. Thus the lifts of $\{b_1,b_2\}$ and $\{c_1,c_2\}$ cross by \myref{alt_exist}.

Next assume that the lifts $\{b_1',b_2'\}$ 
and $\{c_1',c_2'\}$ of $\{b_1,b_2\}$ and $\{c_1,c_2\}$ cross. 
Then by condition (2) of  
\myref{alt_exist} there is a cycle $o'$ of length at most $r$ alternating between 
$\{b_1',b_2'\}$ 
and $\{c_1',c_2'\}$  traversing $\{b_1',b_2'\}$ twice. 
Note that $o'$ contains at most one torso path by \autoref{cut_far}.
Let $o$ be the cycle obtained from the cycle $o'$ by replacing torso paths by paths of $G$ of the 
same length. By \myref{recur_gd0_ONESTEP} the cycle $o$ traverses the local separator 
$\{b_1,b_2\}$ 
twice. 
As the cycle $o$ alternates  between the 
local separators $\{b_1,b_2\}$ and $\{c_1,c_2\}$, these 
two local separators cross by \myref{alt_exist}.
\end{proof}

Sometimes we will omit the term `lift' and simply consider local separators of $G$ as local 
separators of a graph $G'$ obtained by cutting. 
The next lemma says that $r$-local cuttings along non-crossing 2-separators commute.

\begin{lem}\label{cuttings commute}
Let $G$ be an $r$-locally 2-connected graph with distinct non-crossing $r$-local $2$-separators 
$\{a_1,a_2\}$ and $\{b_1,b_2\}$. Then the graphs obtained from $r$-locally cutting these two 
local separators in either order are identical.
\end{lem}

\begin{proof}
 Let $G'$ be the graph obtained from $G$ by $r$-locally cutting $\{a_1,a_2\}$.
 Let $G''$ be the graph obtained from $G'$ by $r$-locally cutting (the lift of) $\{b_1,b_2\}$.
 Let $G_2$ be the graph obtained from $G$ by first cutting $\{b_1,b_2\}$ and then (the lift of) 
$\{a_1,a_2\}$.
 Let $\{b_1',b_2'\}$ be the lift of $\{b_1,b_2\}$ in $G'$. By \myref{lifting-lem}, the slices
 of $\{b_1,b_2\}$ in $G''$ and $G_2$ are identical. By symmetry, the same is true for slices of 
$\{a_1,a_2\}$. Vertices that are not slices are identical by construction. This defines a bijection 
between the vertices of the graphs $G''$ and $G_2$. It is straightforward to check that this 
bijection between the vertices extends to a bijection between the edges. 
\end{proof}

Given a set $\Scal$ of $r$-local 2-separators of a graph $G$ that pairwise do not cross, 
we say that a graph $G'$ is obtained from $G$ by \emph{$r$-locally cutting $\Scal$} if $G'$ is 
obtained from $G$ by the following procedure. Pick a linear ordering of the set $\Scal$. Then 
starting with the graph $G$ we cut along the local separators of $\Scal$ in that linear order. By 
\myref{lifting-lem} and \autoref{lift_non-crossing} this is well-defined. 
By \autoref{cuttings commute}, changing the linear ordering does not affect the graph obtained 
by cutting. Hence in the following we shall speak of \emph{the} graph obtained from $G$ by 
cutting along $\Scal$.

Given a graph $G$, by $\Ncal$ we denote the set of all $r$-local 2-separators of $G$ that are not 
crossed by any $r$-local 2-separator.

\begin{thm}\label{minimal_N}
 Assume $G$ is $r$-locally 2-connected, and let $G'$ be the graph obtained from $G$ by $r$-locally 
cutting $\Ncal$. Then every connected component of $G'$ is $r$-locally 3-connected or a cycle of 
length at most $r$. 

Let $G''$ be a graph obtained from $G$ by $r$-local cuttings such that all connected components 
are $r$-locally 3-connected or cycles of 
length at most $r$. Then in the construction of $G''$ one has to cut at any local separator $X$ 
or one of its lifts for all $X\in \Ncal$. 
\end{thm}

\begin{proof}
Fix a linear ordering of the set $\Ncal$ and let $G_i$ be the graph obtained from $G$ by 
$r$-locally cutting the first $i$ elements of $\Ncal$. By \autoref{loc2con_pres} applied 
recursively each graph $G_i$ is $r$-locally 2-connected. The graph $G'$ is the last graph $G_i$. 

Suppose for a contradiction that some connected component of the graph $G'$ is neither
$r$-locally 
3-connected nor a cycle of length at most $r$. Then by \autoref{structure_torso}, the graph $G'$ 
has an $r$-local 2-separator $\{v,w\}$ that is not crossed by any other $r$-local 2-separator of 
$G'$. 
By \myref{recur_gd0_ONESTEP} applied recursively, $\{v,w\}$ is also an $r$-local 2-separator of 
the graph $G$. 
By the construction of the graph $G'$, the local separator $\{v,w\}$ is not in the set $\Ncal$. 
Thus there is some $r$-local 2-separator $\{a,b\}$ of $G$ 
that crosses $\{v,w\}$. By the choice of the set $\Ncal$, the local 2-separators $\{v,w\}$ and 
$\{a,b\}$ are not crossed by any local separator of $\Ncal$. Hence we can apply 
\myref{lifting-lem} recursively to deduce that $\{a,b\}$ is a local separator of the graph 
$G'$. Applying \autoref{lift_non-crossing} recursively yields that $\{v,w\}$ and $\{a,b\}$ are 
crossing in the graph $G'$. This is a contradiction to the fact that $\{v,w\}$ ia not crossed in 
$G'$. 
Thus every connected component of the graph $G'$ is $r$-locally 
3-connected or a cycle of length at most $r$.

 Now let $G''$ be a graph obtained from $G$ by $r$-local cuttings such that all connected 
components 
are $r$-locally 3-connected or cycles of 
length at most $r$.
 Let $(G_i)$ be a sequence of graphs starting with $G_0=G$ and ending with $G_n=G''$ such that 
$G_{i+1}$ is obtained from $G_i$ by $r$-local cutting. By \autoref{loc2con_pres} applied 
recursively each graph $G_i$ is $r$-locally 2-connected.

Suppose for a contraction there is an $r$-local 2-separator $\{v,w\}$ of the set $\Ncal$ 
such that neither it nor any of its lifts is identical to a local 2-separator at which we locally 
cut to obtain $G_{i+1}$ from $G_i$. 

We will show by induction that  $\{v,w\}$ does not cross any local separator of any graph $G_i$.
Assume we have shown it does not cross any local separator of a graph $G_j$. 
By 
\myref{recur_gd0_ONESTEP}, all local 2-separators of $G_{j+1}$ are lifts of local 2-separators of 
$G_j$. Applying this recursively yields that all local 2-separators at which we locally cut are 
lifts of 
local 2-separators of the graph $G$. So by 
\autoref{lift_non-crossing}  $\{v,w\}$ does not cross any local separator of the 
graphs $G_{j+1}$.  This completes the 
induction step. 

Thus the above is also true for the last graph $G_n=G''$, all of whose connected components are 
$r$-locally 3-connected or cycles of size at most $r$ by assumption. Such graphs do not have an 
$r$-local 2-separator that is not crossed. Hence $\{v,w\}$ cannot exist. This is a contraction. So 
for any local separator of $\Ncal$, it or one of its lifts has to appear in the construction of 
$G''$.   
\end{proof}

\begin{proof}[Proof of \autoref{main_2sepr-intro}.]
 This theorem is a direct consequence of \autoref{minimal_N}. 
\end{proof}

\section{Graph-Decompositions}\label{sec:graph-deco}
The purpose of this section is to define graph-decompositions and explain why they can be 
understood as a generalisation of tree-decompositions with the decomposition-tree replaced by a 
general graph. See also \cite{diestel2005graph}.

 \addcontentsline{toc}{subsection}{Dfn: Identification along}
\begin{dfn}
 First we need some preparation.
Given a graph $G$, a graph $F$ and a family $\Fcal$ of subgraph-embeddings of the graph $F$ in $G$, 
the graph obtained from $G$ by  \emph{identifying along} 
$\Fcal$ is the graph obtained from $G$ by identifying all elements of $\Fcal$. 
In this paper, all embeddings of the graph $F$ will be vertex-disjoint, although we do not make 
this part of the definition. 
Formally, the 
vertex set of this new graph is the vertex set of $G$ 
modulo the equivalence relation generated by the relation where two vertices $v_1$ and $v_2$ are 
related if there are graphs $F_1,F_2\in \Fcal$ with $v_i\in F_i$ (for $i=1,2$) such 
that after applying the isomorphisms to $F$ the vertices $v_1$ and $v_2$ are equal to the 
same vertex of $F$. 
The edges of the quotient-graph are the edges of $G$, where the endvertices are the 
equivalence classes of the 
original endvertices of $G$ -- with the following exception: if two vertices $v$ and $w$ in $F$ are 
joined by edge, then in the quotient-graph we keep only one copy of all the edges of $G$ 
between the copies of $v$ and $w$ in the graphs $F'\in \Fcal$. This completes the 
definition of gluing. Examples of gluing are given in \autoref{fig:gluing1} and 
\autoref{fig:gluing2}.
\end{dfn}

   \begin{figure} [htpb]   
\begin{center}
   	  \includegraphics[height=4cm]{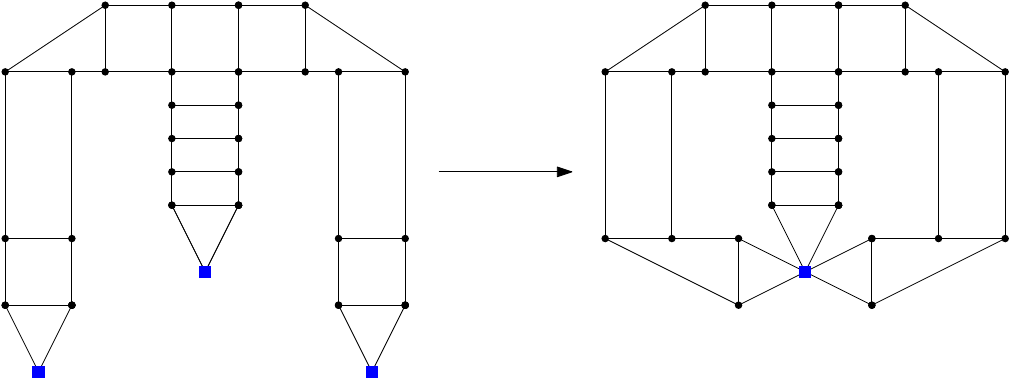}
   	  \caption{An example of a gluing. Here the graph $F$ consists of a 
single vertex and the family $\Fcal$ has three members, which are marked in 
blue. The graph $G$ before the gluing is depicted on the left. The graph after 
the gluing is depicted on the right. }\label{fig:gluing1}
\end{center}
   \end{figure}

   \begin{figure} [htpb]   
\begin{center}
   	  \includegraphics[height=4cm]{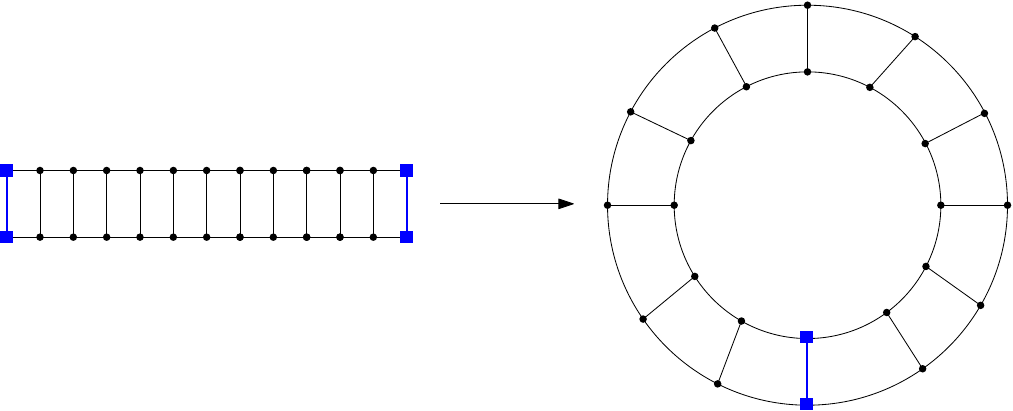}
   	  \caption{An example of a gluing. Here the 
graph $F$ consists of two vertices joined by an 
edge and the family $\Fcal$ has two members, which are marked in 
blue. The graph $G$ before the gluing is depicted on the left. The graph after 
the gluing is depicted on the right. If the graph $F$ consisted just of the two 
vertices without the edge, the gluing would be the graph obtained by the graph on 
the right by adding an edge in parallel to the blue edge.  (This figure is the same as 
\autoref{fig:gluing_new}.)}\label{fig:gluing2}
\end{center}
   \end{figure}

\begin{rem}
Even if the graph $G$ has no loops or parallel edges, graphs obtained from $G$ by identifying along 
a family may have loops or parallel edges, see \autoref{fig:gluing2}. 
\end{rem}

\begin{dfn}[Graph-decomposition]
A \emph{graph-decomposition} consists of a bipartite graph $(B,S)$ with bipartition classes $B$ 
and $S$, where the elements of 
$B$ are referred to as `bags-nodes' and the elements of $S$ are referred to as 
`separating-nodes'. 
This bipartite graph is referred to as the `decomposition graph'. 
For each node $x$ of the decomposition graph, there is a graph $G_x$ associated to $x$. 
Moreover for every edge $e$ of the decomposition graph from a separating-node $s$ to a bag-node 
$b$, 
there is a map $\iota_e$ that maps the associated graph $G_s$ to a subgraph of the associated graph 
$G_b$. We refer to $G_s$ with $s\in S$ as a \emph{local separator} and to $G_b$ with $b\in B$ as a 
\emph{bag}.

The \emph{underlying graph} of a graph-decomposition $(G_x|x\in V(B,S))$ is constructed from the 
disjoint union of the bags $G_b$ with $b\in B$ by identifying along all the families 
given by the copies of the graphs $G_s$ for $s\in S$. 
Formally, for each separating-node $s\in S$, its family is 
$(\iota_e(G_s)))$, where the index ranges over the edges of $(B,S)$ incident with $s$. 
 \addcontentsline{toc}{subsection}{Dfn: Graph-decomposition}
\end{dfn}

\begin{rem}
For some applications, it may be attractive to not construct the underlying graph all at once, as 
we do here, but rather \lq step by step\rq\ by doing the gluing in a certain order. We expect 
that one essentially gets the same graph in the natural cases, independently of the ordering. As 
gluing is the reverse operation to local splitting and for local splitting we have proved 
commutativity (compare \autoref{cuttings commute}), for the graph-decompositions related to 
\autoref{thm:main_intro} this is indeed the case. However, we require the technical tool of contacts 
to formalise this.
We shall not use this in this paper.
\end{rem}

Now we perform the identification for all these families separately. We remark that different 
orderings in which we perform these identification result in the same graph.

The \emph{width} of a graph-decomposition $(G_x|x\in V(B,S))$ is the maximal vertex number of a bag 
$G_b$ with 
$b\in B$  take away one\footnote{This convention to take away one is common in the literature.  
Consequently, trees have tree-width one.}. The \emph{adhesion} of a graph-decomposition 
$(G_x|x\in V(B,S))$ is the maximum vertex 
number of a local separator $G_s$ with $s\in S$. 
This completes the definition of graph-decompositions and related concepts.

\begin{eg}
Essentially, tree-decompositions are examples of graph-decompositions. 
 Indeed, given a tree-decomposition, one obtains a new tree-decomposition by subdividing every edge 
once and associating to each new vertex the separator associated to that edge (this separator is 
given by taking the intersection of the two bags at the endvertices of that edge).

This defines a graph-decomposition whose decomposition-graph is the decomposition-tree of this 
newly constructed tree-decomposition. Its bags are the original bags of that tree-decomposition 
and its local separators are separators of the old tree-decomposition; that is, the bags associated 
to the new vertices of the new tree-decomposition.

The notions of width and adhesion as defined above for graph-decompositions whose decomposition 
graphs are trees coincide with the 
standard notions for tree-decompositions when interpreted as graph-decompositions in the way 
explained above. 
\end{eg}

\begin{eg}
 The graph on the right of \autoref{fig:gluing1} has a graph-decomposition with only one bag which 
is given by the graph on the left, and only one local separator, which is given by the blue vertex. 
Its decomposition-graph is the bipartite graph consisting of three edges in parallel.
\end{eg}

\begin{eg}
 The graph on the right of \autoref{fig:gluing2} has a graph-decomposition with only one bag which 
is given by the graph on the left, and only one local separator, which consists of an edge. 
Its decomposition-graph is the bipartite graph consisting of two edges in parallel.
\end{eg}

 \addcontentsline{toc}{subsection}{Dfn: Local cut}
\begin{dfn}
 This definition is slightly technical. 
For this 
definition consider graph-decompositions such that for every separating-node $s$ 
there are no two distinct embedding maps $\iota_{f}$ that map vertices of $s$ to the same vertex or 
endvertices of the same edge (informally, this means that the images of $s$ under the maps 
$\iota_{f}$ are sufficiently far away from each other. For sets $s$ we are interested in here  
-- those that come from $r$-local $2$-separators -- this will always be the case).

Given a graph-decomposition of a graph $G$ and an edge $f$ of the decomposition-graph incident with 
a separating-node $s$ and a bag-node $b$, the \emph{local cut} associated to $f$ consists of those 
edges 
of $G$ that have precisely one endvertex in $\iota_f(G_s)$ and the other endvertex in the bag 
$\iota_f(G_b)$. 
We say that a cycle 
\emph{traverses the local cut associated to the edge $f$ oddly} if it contains an odd number of 
edges from that local cut. 
\end{dfn}

 \addcontentsline{toc}{subsection}{Dfn: Locality}
\begin{dfn}
 A graph-decomposition has \emph{locality} $r$ if every cycle traversing a local cut 
corresponding to an edge of the decomposition-graph oddly\footnote{An alternative definition would 
be to replace `traversing oddly' 
by `traversing effectively zero', taking additionally orientations of traversals into account. For 
simplicity we just make the definition with `oddly'.} has length larger than $r$.
\end{dfn}

\begin{eg}
In graph-decompositions whose decomposition graph is a tree all cycles traverse evenly and hence 
their locality is infinite. 
\end{eg}

\begin{eg}
The locality of the graph-decomposition described in \autoref{fig:gluing2} is 11, as the 
shortest cycle traversing oddly has length 12. 
The locality of the graph-decomposition described in \autoref{fig:gluing1} is 9.
\end{eg}

There is a correspondence between nested sets of separations and tree-decompositions. A 
corresponding fact for graph-decompositions is also true. Here we only need the following special 
case of this correspondence. 

\begin{lem}\label{nested_to_td}
Let $G$ be a graph with a set $S$ of non-crossing $r$-local 2-separators.
Let $B$ be the set of connected components of the graph obtained from $G$ 
by $r$-locally cutting along 
$S$.

Then there is a decomposition graph with bipartition $(B,S)$ of a graph-decomposition of $G$ of 
adhesion two and locality $r$. 
\end{lem}

\begin{proof}
For every local component of a local separator $G_S$ with $s\in S$ we have an edge $e$ in the graph 
$(B,S)$ 
from $s$ to the unique bag-node $b\in B$ such that $G_b$ contains the slices of $\iota_e(G_s)$ 
for that local 
component. The map 
$\iota_e$ maps $G_s$ to this copy of $G_s$ in $G_b$. So $(B,S)$ is the decomposition graph 
of a graph-decomposition of $G$. It has adhesion two as all elements  $G_s$ with $s\in S$ have size 
two. It has locality $r$ as all local separators $G_s$ with $s\in S$ are $r$-local. 
\end{proof}

A \emph{torso} of a bag $G_b$ of a graph-decomposition is obtained from $G_b$ by joining for 
every 
map 
$\iota$ from a local separator $G_s$ any two vertices in the image of $\iota$ by 
an 
edge whenever $sb$ is an edge of the decomposition graph. 

\begin{proof}[Proof of \autoref{thm:main_intro}.]
 Combine \autoref{nested_to_td} and \autoref{minimal_N}. 
\end{proof}

\section{Outlook}\label{sec:out}

Having finished the proof of the local 2-separator theorem, we outline possible ways in 
which this theorem could be extended. 
We continue our investigation of local separators in \cite{locksepr} by 
proving a 
local version of the tangle tree theorem. 
Another direction, might be to prove a matroidal analogue of our local $2$-separator theorem.

\begin{que}
 Can you prove a local 2-separator theorem for (representable) 
matroids that is reminiscent of \autoref{thm:main_intro}?
\end{que}

A natural next step would be to prove a 
local 
version of the Grohe-Decomposition-Theorem, which gives a decomposition of a 3-connected graph into 
`quasi 
4-connected components' \cite{grohe2016quasi}. We hope that with the methods of this paper one 
should be able to 
prove 
the following. 

It is natural to try to extend the notion of local 2-separators to local 3-separators using 
explorer-neighbourhoods around three vertices; this is work in progress \cite{locksepr}. 
As for local 2-separators the intuition for this comes from separators of the $r$-local cover. I 
think that understanding local separators through separators of the $r$-local cover is an 
exciting area for future research. 
Given a 
parameter $s\in 
\Nbb\cup\{\infty\}$, we say that a graph 
is \emph{$s$-locally quasi 4-connected} if for every $s$-local 3-separator all but one of its 
components (that is, components of the punctured explorer-neighbourhood) contain at most one 
vertex. 

\begin{con}\label{grohe_ext}
For every parameter $r\in \Nbb\cup\{\infty\}$ there is a parameter $s$ such that  
every $s$-locally $3$-connected graph $G$ has a graph-decomposition
of locality at least $r$ and adhesion at most $3$ such that all its torsos are 
 minors
of $G$ that are either $r$-locally quasi-4-connected or a complete graph of order at most 4.
\end{con}

\begin{rem}
 We do not conjecture any relationship between $r$ and $s$, and even $r=s$ may be possible.
\end{rem}

\appendix
\section{Appendix I: an alternative proof for the existential statement}\label{ex2}

In this part we give an alternative proof of the first sentence of \autoref{main_2sepr-intro} that 
only relies on lemmas proved before \autoref{props}.

Basic examples of $r$-locally 2-connected graphs are cycles of length at 
most $r$ and $r$-locally 
3-connected graphs. The following theorem says that all  $r$-locally 2-connected graphs 
can be constructed from these basic graphs by $r$-local 2-sums.

   \begin{thm}\label{existential}
   Given a parameter $r$, every $r$-locally 2-connected graph $G$ can be obtained
   from a graph that is a disjoint union of 
   $r$-locally 3-connected graphs and cycles of length at most $r$ by $r$-local 
2-sums. 
   \end{thm}

\begin{proof}
In order to simplify the presentation of this proof, we work with the slight modification of the 
definition of \lq $r$-local cutting', where we leave out the step, where we replace torso edges 
by torso paths.  
 
Our strategy to prove this theorem is the following. 
If the graph $G$ does not have any $r$-local 2-separator, it is $r$-locally 3-connected and we are 
done. Otherwise, pick an $r$-local 2-separator arbitrarily and 
$r$-locally cut along that local 2-separator. Then we iterate this procedure.

Formally, this is expressed as follows. Let $G_1=G$. Assume we already defined $G_i$. If $G_i$ is 
$r$-locally 3-connected or a cycle of length at most $r$, we stop. Otherwise, pick an $r$-local 
2-separator $\{a_i,b_i\}$ 
 arbitrarily, and obtain the graph $G_{i+1}$ by $r$-locally cutting at 
the local separator $\{a_i,b_i\}$. 
It is easily proved by induction that each graph $G_i$ is $r$-locally 2-connected using 
\autoref{loc2con_pres}. 

If this procedure terminates, then by \autoref{inverse_sum_cut} applied recursively, the graph 
$G$ has the desired 
decomposition. 

Hence all that remains to show is that this procedure stops eventually. For that we introduce 
parameters that decrease in each cutting operation. 
Let $\gamma_i$ be the dimension of the cycle space of the graph $G_i$. 
We abbreviate: $e_i=|E(G_i)|$ and $v_i=|V(G_i)|$. By $k_i$ we denote the number of components of 
the graph $G_i$. 
The following identity relating edge number, vertex number and the number of components of a graph 
is well-known (see for example \cite{DiestelBookCurrent}):
\begin{equation}\label{eq1}
 e_i=\gamma_i+v_i-k_i
\end{equation}

Let $\ell$ be the number of torso edges produced at the $i$-th cutting step. We have: 

\[
 e_{i+1}=e_i+\ell, \ \ v_{i+1}= v_i+2(\ell-1)
\]

Subtracting \autoref{eq1} from itself with indices `$i+1$' and `$i$', respectively, and plugging in 
the above yields:
\begin{equation}\label{eq2}
 \gamma_{i+1}=\gamma_i -(\ell-2)+(k_{i+1}-k_i)
\end{equation}

\begin{rem}[Motivation]
First we explain heuristically why one expects this process to terminate referring to 
\autoref{eq2}. 
If $G_{i+1}$ has more components than $G_i$, then one expects that each of these new components is 
smaller than $G_i$ and hence one could apply induction.

If $\ell>2$, then one can apply induction on $\gamma_i$. Hence it remains to consider the case that 
$\ell=2$ and the graphs $G_{i+1}$ and $G_i$ have the same numbers of components; that is, 
$k_{i+1}-k_i=0$. Still we would like to apply induction on $\gamma_i$ but it might not go down 
immediately -- but it will go down eventually. To see that consider 
$\overline{\gamma}_i$, which is defined to be the dimension of the cycles generated by those of 
length at 
most $r$. We will show  that $\overline{\gamma}_{i+1}>\overline{\gamma}_i$ in this case.

As $\overline{\gamma}_i\leq \gamma_i$, we 
can make only a bounded number of steps in which $\gamma_i$ stays constant -- as 
$\overline{\gamma}_i$ 
increases in each of these steps. Thus $\gamma_i$ will 
decrease eventually and we may apply induction.
\end{rem}

Formally, we argue like this. 
Given a connected graph $G$, we consider the triple $(\gamma, -\overline{\gamma}, v)$, where 
$\gamma$ is the 
dimension of the cycle space of $G$, $\overline{\gamma}$ is the dimension of its cycles generated 
by 
cycles 
of length at most $r$, and $v$ is the number of vertices of $G$. 
We consider the order on connected graphs given by the lexicographical ordering according to 
the triples:
$(\gamma, -\overline{\gamma}, v)$; that is, if $\gamma(H)<\gamma(G)$, then $H<G$. If the 
$\gamma$-values 
are 
equal, we compare the values for $-\overline{\gamma}$. If they are also equal, we compare the 
vertex-numbers. We refer to this ordering as the \emph{triplex-ordering}.

\begin{sublem}\label{well-founded}
 The triplex-ordering is well-founded.
\end{sublem}

\begin{proof}
Suppose for a contradiction there is an infinite sequence that is strictly decreasing in the 
triplex ordering.
Then the parameter $\gamma$ must be eventually constant. As $\overline{\gamma}\leq \gamma$, there 
are only boundedly many possible values for $-\overline{\gamma}$ (after that). Hence this parameter 
also has to be eventually constant. Then also the vertex number has to be eventually constant. This 
is a contradiction to the assumption that the sequence is infinite and strictly decreasing. 
\end{proof}

\begin{sublem}\label{cases}
 Let $G$ be a connected graph and $G'$ be obtained from $G$ by $r$-locally cutting an $r$-local 
2-separator of $G$, then every connected component of $G'$ is strictly smaller than $G$ in the 
triplex-ordering.
\end{sublem}

\begin{proof}
{\bf Case 1: } the graph $G'$ is connected.
If the number $\ell$ of slices is at least three, then by \autoref{eq2}, $G'$ is strictly smaller 
in the triplex-ordering. So we may assume, and we do assume, that $\ell=2$. Then by \autoref{eq2} 
$\gamma(G')=\gamma(G)$. Hence it suffices to do the following computation.

\begin{sublem}\label{gamma_bar}
 $\overline{\gamma}(G')>\overline{\gamma}(G)$. 
\end{sublem}

\begin{proof}
We denote by $\Fbb_2^E$ the vector space over the finite field $\Fbb_2$ whose set of coordinates is 
$E$, the set of edges of the graph $G$. Similarly, we denote by $\Fbb_2^{E'}$ the vector 
space over the finite field $\Fbb_2$ whose set of coordinates is $E'$, the set of 
edges of the graph $G'$. We denote by $\Ccal$ those vectors generated by the cycles of length at 
most $r$ in the graph $G$. We denote by $\Ccal'$ those vectors generated by the cycles of length 
at 
most $r$ in the graph $G'$. The set $E'$ is obtained from the set $E$ by adding the two torso 
edges. Consider the set $\Ccal''$ of vectors $v$ with coordinates in $E$ such that there are 
assignments to the torso edges so that $v$ extends to a vector in $\Ccal'$. Note that $\Ccal''$ is 
a vector space that has the same dimension as $\Ccal'$. 

Next we will show that $\Ccal\se \Ccal''$. For that take a cycle $o$ of the graph $G$ of length at 
most $r$. If its edge set is a cycle of the graph $G'$, the cycle $o$ is also in $\Ccal''$.
Otherwise, it contains a vertex of the local separator. So it is contained in a ball 
of radius $r/2$ around that vertex. If it does not contain the other vertex of the local separator, 
it is also a cycle of $G'$. Otherwise, in $G'$ it consists of two 
paths joining the slices of that local separator. This edge set is also in the vector 
space $\Ccal''$.

By \myref{local_is_very_local} the vector space $\Ccal''$  contains a path between two slices, 
which is not in $\Ccal$ (indeed such a path is a subpath of the cycle $o'$ from that lemma).
Thus the dimension of the vector space $\Ccal''$ is strictly larger than the dimension of the 
vector space $\Ccal$. So the dimension of the vector space $\Ccal'$ is strictly larger than the 
dimension of the 
vector space $\Ccal$.
\end{proof}

Hence by \autoref{gamma_bar},
$\overline{\gamma}(G')>\overline{\gamma}(G)$, and so $G'$ is strictly 
smaller than $G$ in the triplex ordering in this case.

{\bf Case 2: } the graph $G'$ is disconnected. Let $H'$ be an arbitrary component of the graph 
$G'$. 
As each connected component of the graph $G'$ contains a cycle through one of its torso edges by 
\myref{local_is_very_local}, by \autoref{eq2} we have that $\gamma(H')\leq \gamma(G)$. Moreover, 
if we have equality, all other components of $G'$ are single cycles (as they don't have local 
cutvertices by \autoref{loc2con_pres}), and $\ell=2$. So $\overline{\gamma}(H')= 
\overline{\gamma}(G)$. By 
\myref{local_is_very_local} each component of $G'$ has at least one vertex that is not a 
slice (note that as $\ell=2$ local cutting does not produce an artificial cycle of 
length two). 
Hence $v(H')<v(G)$, and so $H'$ is strictly smaller than 
$G$ in the triplex-ordering. This completes the proof in Case 2 of this sublemma. 
\end{proof}

By \autoref{cases}, each connected component of the graph $G_{i+1}$ is strictly smaller than some 
connected component of the graph $G_i$ in the triplex-ordering. Hence by \autoref{well-founded} we 
may apply induction. Thus this 
procedure has 
to stop eventually. Let $\Gcal$ be the set of connected components of the graph $G_i$ where this 
terminates. As explained above, we can apply \autoref{inverse_sum_cut} recursively to deduce that 
the graph $G$ is constructed via local 2-sums from the set $\Gcal$. 
\end{proof}

\begin{rem}
 It seems to us that \autoref{existential} is also true with a different notion of local 
2-separators that is 
based on double-balls. However, examples such as that given in \autoref{fig:no-corner} show that 
such alternative decompositions cannot be unique. Thus 
we believe that such a statement would be less applicable than 
\autoref{existential}, see Application (D) in the Introduction for details. 
\end{rem}

\bibliographystyle{plain}
\bibliography{literatur}

\end{document}